\documentclass[doublecolumn,letterpaper,10pt,conference]{ieeeconf}
\let\labelindent\relax

\usepackage{breakcites}
\usepackage{amsmath,amssymb,amsfonts}
\usepackage{graphicx}
\usepackage{textcomp}
\usepackage{styles/style}
\usepackage{nccmath}
\usepackage{indentfirst}
\usepackage{rotating}
\usepackage{algorithm,algorithmic}
\usepackage{pgf}
\usepackage{tikz}
\usepackage{tikz-3dplot}
\usepackage{pifont}
\usepackage{aligned-overset}
\usepackage{longtable}
\usepackage{booktabs}
\usepackage{nicefrac}
\usepackage[usenames,dvipsnames]{pstricks}
\usetikzlibrary{arrows,shapes,automata,decorations.text,decorations.pathmorphing,backgrounds,positioning,fit}
\usepackage{enumitem}
\usepackage{color}
\usepackage{xcolor}
\usepackage{tabularx}
\usepackage{pgfplots}
\usepackage{pst-grad} 
\usepackage{pst-plot} 
\usetikzlibrary{positioning}
\usepackage{rotating}
\usepackage{cite}
\usepackage[colorlinks=true, allcolors=blue]{hyperref}
\usepackage{dblfloatfix}
\usepackage{url,amsmath,setspace,amssymb}
\usepackage{tikz}
\usepackage{graphicx}
\usepackage{caption}
\usepackage{subcaption}
\usepackage{float}
\usepackage{soul}
\usepackage{epsfig}
\usepackage{longtable}
\usepackage{relsize}
\allowdisplaybreaks
\IEEEoverridecommandlockouts
\usetikzlibrary{external}
\definecolor{royalblue}{RGB}{65, 105, 225}
\definecolor{seagreen}{RGB}{46, 139, 87}
\definecolor{firebrick}{RGB}{178,34,34}
\definecolor{darkviolet}{RGB}{138, 43, 226}
\definecolor{carrotorange}{RGB}{237, 145, 33}
\pgfplotsset{every tick label/.append style={font=\small}}

\title{\LARGE \bf On Arbitrary Compression for Decentralized Consensus and Stochastic Optimization over Directed Networks}
\author{Mohammad Taha Toghani and C\'{e}sar A. Uribe
	\thanks{The authors are with the Department of Electrical and Computer Engineering, Rice University, 6100 Main St, Houston, TX 77005, USA, \{\href{mailto:mttoghani@rice.edu}{mttoghani}, \href{mailto:cauribe@rice.edu}{cauribe}\}@rice.edu. This work was partially funded by \mbox{ARPA-H} Strategic Initiative Seed Fund \#916012.}
}

\begin{document}
	\maketitle
	
	\begin{abstract}
		We study the decentralized consensus and stochastic optimization problems with compressed communications over static directed graphs. We propose an iterative gradient-based algorithm that compresses messages according to a desired compression ratio. The proposed method provably reduces the communication overhead on the network at every communication round. Contrary to existing literature, we allow for arbitrary compression ratios in the communicated messages. We show a linear convergence rate for the proposed method on the consensus problem. Moreover, we provide explicit convergence rates for decentralized stochastic optimization problems on smooth functions that are either (i) strongly convex, (ii) convex, or (iii) non-convex. Finally, we provide numerical experiments to illustrate convergence under arbitrary compression ratios and the communication efficiency of our algorithm.
	\end{abstract}


	\section{Introduction}
	\label{sec:introduction}
	In this work, we consider to solve the following stochastic optimization problem over a directed network of $n$ nodes:
	\begin{align}\label{eq:opt}
		f^\star &\coloneqq \min_{\vx{\in}\bbR^d} \left[ f(\vx) \coloneqq \frac{1}{n} \sum_{i=1}^n f_i(\vx) \right],\\
		f_i(\vx) &\coloneqq \bbE_{\xi_i\sim\mcD_i} \tilde{f}_i(\vx,\xi_i), \quad \text{for all } i{\in}[n],\nonumber
	\end{align}
	where $\tilde{f}_i:\bbR^d{\times} \mcS_i \rightarrow \bbR$ is agent $i$'s loss function with local data distribution $\mcD_i$ and $\mcS_i$ the output space of the random variable $\xi_i$. Problem~\eqref{eq:opt} can be reformulated as a constrained problem with the following format:
	\begin{align}\label{eq:opt-dec}
		\min_{(\vx_1,\dots,\vx_n){\in}\mcX} &\left[\frac{1}{n} \sum_{i=1}^n f_i(\vx_i) \right],\\
		\mcX = \{(\vx_1,\vx_2,\dots,\vx_n)&{\in}(\bbR^d)^n\,\, \text{s.t.}\,\,\vx_1 =\vx_2 = \dots =\vx_n\},\nonumber
	\end{align}
	which demands to jointly address consensus and optimization simultaneously. In the decentralized optimization setup, agents are only allowed to exchange information via a communication graph. Given the specific function choice
	\begin{align}\label{eq:consensus}
		f_i(\vx) = \lVert \vx-\vx_i(0) \rVert^2,
	\end{align}
	Problem~\eqref{eq:opt} turns into the average consensus problem~\cite{xiao2004fast} with initial vectors $\vx_i(0)$.
	
	Problem~\eqref{eq:opt} has been thoroughly studied under distributed~\cite{konevcny2015federated,spiridonoff2021communication,reddi2020adaptive} and decentralized~\cite{nedic2009distributed,nedic2017achieving} communication setups. The distributed setup provides benefits such as data privacy, on-device training, and parallel computation~\cite{kairouz2019advances}. This however imposes challenges like communication bottlenecks~\cite{konevcny2015federated}, scalability~\cite{Olshevsky2017LinearTA}, and vulnerability to adversarial attacks~\cite{su2016fault}.
	
	The decentralized optimization problem over undirected networks has been studied in~\cite{Nedic2009DistributedSM,ram2010distributed,shi2015extra,uribe2021dual}. The core of these algorithms relies on balanced communications, which usually translates into a doubly stochastic mixing matrix associated with the graph. Kempe et al.~\cite{kempe2003gossip} suggested the push-sum technique for consensus over (strongly connected) directed networks. Moreover, studies in~\cite{tsianos2012push,nedic2016stochastic,nedic2015nonasymptotic,assran2019stochastic} consider the inference, convex, and non-convex optimization over directed networks using the push-sum idea.
	
	Decentralized algorithms classically require the agents to share all their parameters with their local neighbors at each round. Hence, communication bottlenecks might appear in the networks. Works in~\cite{rabbat2005quantized,duchi2011dual,reisizadeh2019exact,xian2021communication} are proposed to address large communication requirements. More recently, using error-feedback techniques, several efforts have been made to mitigate the communication burden for consensus~\cite{zhang2021innovation,koloskova2019decentralized,toghani2021scalable}, inference~\cite{toghani2021communication,Toghani2021CommunicationEfficientAF}, and optimization problems~\cite{koloskova2019decentralized,koloskova2019decentralized2,kovalev2021linearly} over undirected graphs.
	
	Taheri et al.~\cite{taheri2020quantized} extended the results in~\cite{koloskova2019decentralized,koloskova2019decentralized2} to directed communications. However, the minimum compression ratio allowed in their proposed algorithm is bounded from below. For example, topologies such as Ring with spectral gap $\mcO(n^{-2})$ require a compression ratio of order $\Omega(n^4/(n^4+1))$, which is close to~$1$ (no compression) for large $n$. \textit{The main objective of this work is to allow arbitrary compression for decentralized consensus and stochastic optimization over directed networks.}

	In this work, inspired by \cite{zhang2021innovation,koloskova2019decentralized,koloskova2019decentralized2}, we consider a \textit{consensus stepsize} $\gamma$ to enable arbitrary compression ratios. In a nutshell, we consider solving Problem~\eqref{eq:opt-dec} over a static directed network with arbitrary compression. We summarize our contributions as follows:
	\begin{itemize}[leftmargin=1em]
		\item We propose an algorithm for decentralized consensus and stochastic optimization over directed communication graphs with compressed communications.
		\item Under an arbitrary compression ratio $\omega\in(0,1]$, we show explicit convergence rates of our algorithm for smooth functions which are either (i) strongly convex, (ii) convex, or (iii) non-convex. We also provide a comprehensive comparison of our results with their counterparts for directed and undirected graphs.
		\item We present empirical results to highlight the convergence under arbitrary compression and the communication efficiency of our method.
	\end{itemize}

	The rest of the paper is organized as follows. In Section~\ref{sec:setup}, describing the problem setup, we present our algorithm and highlight the results. Section~\ref{sec:analysis} presents our theoretical guarantees, and Section~\ref{sec:experiments} contains numerical experiments that corroborate the communication efficiency and arbitrary compression of the proposed method. We conclude with the remarks and future works in Section~\ref{sec:conclusion}.

	\noindent\textbf{Notation:} We show vectors and matrices with boldface lower-case and upper-case letters respectively. We refer to the entry in the $i$-th row and $j$-th column of matrix $\mX$ with $\mX_{ij}$. We use $[\mX]_i$ ($[\vx]_i$) to indicate the $i$-th row (element) of matrix $\mX$ (vector $\vx$). We refer to the vector of all one with size $n$ and identity matrix, respectively with $\vect{1}$ and $\mI$. For a parameter $\vx$, we write \mbox{$\vx(t)$} in reference to its value at time $t$. We refer to agents by subscripts. We also denote \mbox{$\lVert\vx\rVert$} and \mbox{$\lVert \mX \rVert_F$} respectively as $2$-norm of vector $\vx$ and Frobenius norm of matrix $\mX$.  We use \mbox{$\lVert \mW \rVert$} to indicate the matrix norm of a square matrix $\mW$.


	\section{Problem Setup, Algorithm, \& Results}\label{sec:setup}
	This section introduces the problem setup and presents the studied algorithm with supporting discussions on the convergence rate for convex and non-convex problems.
	
	\noindent\textbf{$\diamond$ Communication Network:}
	Consider a fixed, directed, and strongly connected network
	\mbox{$\mcG = \{[n],\mcE\}$}, \mbox{$\mcE\subseteq [n]{\times}[n]$}, where $(i,j)\in\mcE$ if there is an edge from node $i$ to $j$. For each agent \mbox{$i\in[n]$}, we define \mbox{$\mcN_i^{-}=\{j \text{ s.t. } (j,i)\in\mcE\}\cup\{i\}$} and \mbox{$\mcN_i^{+}=\{j \text{ s.t. } (i,j)\in\mcE\}\cup\{i\}$} as the sets of in-neighbors and out-neighbors. We consider a column stochastic \textit{mixing matrix} \mbox{$\mW{\in} [0,1]^{n{\times} n}$} (\mbox{$\vect{1}^\top\mW{=}\vect{1}^\top$}) consistent with network \mbox{$\mcG$}, where \mbox{$\mW_{ij} {=} 0$} if $(j,i)\notin \mcE$. For instance, mixing matrix \mbox{$\mW_{ij}=1/|\mcN_{j}^{+}|$}, for all $(j,i)\in\mcE$, is column stochastic. We denote $\beta=\lVert \mW{-}\mI \rVert$. Following~\cite{nedic2016stochastic}, there exists a set of positive constants \mbox{$\delta\in(0,1]$} (spectral gap), \mbox{$C,\kappa>0$}, and a stochastic vector \mbox{$\vphi\in\bbR^{n}$} (\mbox{$\vect{1}^\top \vphi = 1$}) such that the following properties hold for matrix $\mW$: for all \mbox{$t\geq 0$}, $i\in[n]$,
	\begin{align}\label{eq:mixing-matrix}
		\mW\vphi=\vphi,\,\,\, [\mW^t \vect{1}]_i \geq \kappa, \,\,\, \big\lVert\mW^t{-}\vphi\vect{1}^\top\big\rVert\leq C(1{-}\delta)^t.
	\end{align}

	\noindent\textbf{$\diamond$ Compression Operator:} We consider the class of compression operators \mbox{$Q:\bbR^d {\times} \mcZ {\times} (0,1] \rightarrow \bbR^d$} that satisfy
	\begin{align}\label{eq:q-comp}
		\bbE_{\vzeta}\big\lVert Q(\vx,\vzeta,\omega) - \vx \big\rVert^2 \leq (1{-}\omega) \big\lVert \vx \big\rVert^2, \qquad \forall \vx\in\bbR^d,
	\end{align}
	where \mbox{$\omega\in(0,1]$} is the \textit{compression ratio}, and random variable \mbox{$\vzeta$} with output space $\mcZ$ specifies the randomization of the operator. Every time that an agent uses operator $Q$, an independent realization of $\mcZ$ is obtained. We drop the dependencies on \mbox{$\vzeta$} and {$\omega$} from $Q$ and $\bbE$ for simplicity of notation. In the definition above, \mbox{$\omega{=}1$} implies no compression. The property in~\eqref{eq:q-comp}, includes a number of sparsification and quantization operators. For example, operator \mbox{$\mathrm{rand}_{\alpha\%}$} (or \mbox{$\mathrm{top}_{\alpha\%}$}) that selects random (or top) $\alpha$ percent elements out of $d$ entries. Here $\vzeta$ represents the randomness when selecting the entries. Another example is the operator \mbox{$\mathrm{qsgd}_{k}$} that rounds each entry to one of the $2^{k-1}{+}1$ quantized levels. Check~\cite[Table~1]{toghani2021communication} for the number of bits required for each operator. Details of other operators can be found in~\cite{stich2018sparsified,beznosikov2020biased}.

	\noindent\textbf{$\diamond$ Algorithm:} Algorithm~\ref{alg:comp-push-sum} shows the pseudo code of the proposed method. Consider each agent \mbox{$i\in[n]$} initially maintains the set of parameters \mbox{$\vx_i(t)$} and \mbox{$y_i(t)$}, and \mbox{$\hat\vx_j(t)$} for all \mbox{$j\in\mcN_i^{-}$} as the approximation of agent $i$'s in-neighbors parameters. Agent $i$ can only send (receive) messages to (from) $\mcN_i^{+}$ ($\mcN_i^{-}$).
	
	Let us first consider the update for parameters $\vx_i(t)$. At each round $t$, each agent $i$ computes a compressed version of the difference between vector $\vx_i(t)$ and its approximation $\hat\vx_i(t)$ using the compression operator $Q$, and transmits the compressed vector $\vq_i(t)$ to its out neighbors. Then, using the received compressed messages, agent $i$ updates its approximation of the in-neighbors' parameters, $\hat\vx_j(t)$ (Line~\ref{ln:update-x-hat}), and combines the approximated vectors using Line~\ref{ln:update-u}. As a consequence of~\eqref{eq:mixing-matrix}, the mixing matrix corresponding to the directed graph $\mcG$ is column stochastic which implies each agent converges to a weighted ($\phi_i$) version of the consensus vector. Inspired by~\cite{kempe2003gossip,tsianos2012push}, each agent $i\in[n]$ considers a slack scalar $y_i(t)$ initialized to $1$. Since $y_i(t)$ is an scalar variable, we assume there is no need for its compression. Along with compressed vectors $\vq_i(t)$, the agents transmit their $y_i(t)$ and apply the consensus step in Line~\ref{ln:update-y} to update the slack variables. This way, $\vz_i(t)$ indicates the normalized ratio between variables $\vx_i(t)$ and $y_i(t)$ (Line~\ref{ln:update-z}), which is common among all agents. In line~\ref{ln:update-x} of Algorithm~\ref{alg:comp-push-sum}, we show our method for the ``average consensus'', and ``stochastic optimization" problems, where in the case of stochastic optimization, the next iterate $\vx_i(t{+}1)$ will move in the direction of the local stochastic gradient.
	
	\begin{algorithm}[!t]
		\caption{\small Decentralized Consensus / Stochastic Optimization with Arbitrary Compressed Push-Sum over Directed Networks}\label{alg:comp-push-sum}
		{\textbf{input:} initial parameters \mbox{$\vx_i(0) \in \bbR^d$}, \mbox{for all $i{\in}[n]$}, column stochastic mixing matrix $\mW$ consistent with graph $\mcG$, consensus stepsize \mbox{$\gamma \in (0,1]$}, compression operator $Q$ with \mbox{$\omega \in (0,1]$}, optimization stepsize $\eta>0$.}\\
		
		\vspace{-1.2em}
		\begin{algorithmic}[1]
			\STATE{$\hat{\vx}_i(0) \coloneqq \vect{0}$, $y_i(0)\coloneqq 1$, for all  $i{\in}[n]$}
			\FOR{$t$ \textbf{in} $0,\dots,T{-}1$, in parallel for all $i{\in}[n]$}
			\STATE{$\vq_i(t)\coloneqq Q(\vx_i(t) - \hat{\vx}_i(t))$}\label{ln:diff-comp}
			\STATE{send \mbox{$\left(\vq_i(t),y_i(t)\right)$} to $\mcN_i^{+}$}\label{ln:send}
			\STATE{receive \mbox{$\left(\vq_j(t),y_j(t)\right)$} from all $j\in\mcN_i^{-}$}\label{ln:receive}
			\STATE{$\hat{\vx}_j(t{+}1)\coloneqq\hat{\vx}_j(t)+\vq_j(t)$, for all $j\in\mcN_i^{-}$}\label{ln:update-x-hat}
			\STATE{$y_i(t{+}1) \coloneqq \sum_{j{\in}\mcN_i^{-}}  \mW_{ij}\,y_j(t)$}\label{ln:update-y}
			\STATE{$\vu_i(t{+}1) \coloneqq \vx_i(t) + \gamma\hspace{-0.3em}\sum\limits_{\,\,\,\,j{\in}\mcN_i^{-}} \hspace{-0.3em} \mW_{ij}\left(\hat{\vx}_j(t{+}1){-}\hat{\vx}_i(t{+}1)\right)$}\label{ln:update-u}
			\STATE{\vspace{-0.8em}$\vz_i(t{+}1) \coloneqq \frac{\vu_i(t{+}1)}{y_i(t{+}1)} $}\label{ln:update-z}\vspace{0.2em}
			\STATE{$\diamond$ \textcolor{blue}{Option I:} (Average Consensus)\\
				\qquad $\vx_i(t{+}1)\coloneqq\vu_i(t{+}1)$\\
				$\diamond$ \textcolor{blue}{Option II:} (Stochastic Optimization)\\
				\qquad $\vx_i(t{+}1)\coloneqq\vu_i(t{+}1) - \eta \nabla \tilde{f}_i(\vz_i(t{+}1),\xi_{i,t{+}1})$}\label{ln:update-x}
			\ENDFOR
		\end{algorithmic}
	\end{algorithm}
	
	Algorithm~\ref{alg:comp-push-sum} resembles~\cite[Algorithms 1 and 2]{taheri2020quantized}, but the main difference is the existence of a consensus stepsize \mbox{$\gamma\in(0,1]$}. The proper choice of this stepsize enables our method to converge under any arbitrary compression ratio \mbox{$\omega\in(0,1]$}. Table~\ref{tab:cons-comparison} provides a full comparison of the convergence rates and compression ratios for various consensus methods. Figure~\ref{fig:cons-omega} in Section~\ref{sec:experiments} highlights the significance of the consensus stepsize on the convergence under arbitrary compressed communications.

	We now present Algorithm~\ref{alg:comp-push-sum} in matrix notation. Let \mbox{$\mX(t)=\big[\vx_1(t),\dots,\vx_n(t)\big]^\top$}, \mbox{$\overline{\mX}(t)=\big[\overline{\vx}(t),\dots,\overline{\vx}(t)\big]^\top$},\\
	\mbox{$Q(\mX)=\big[Q(\vx_1),\dots,Q(\vx_n)\big]^\top$}, where \mbox{$\overline{\vx}(t) {=} \big(\mX(t)^\top\vect{1}\big)/n$}.\\ We also define term \mbox{$\partial F(\mX{(t)}) = \bbE[\partial\tilde{F}(\mX{(t)},\vxi_{t})]$}, where \mbox{$\partial\tilde{F}(\mX{(t)},\vxi_{t}) = [\nabla\tilde{f}_1(\vx_1(t),\xi_{1,t}),\dots,\nabla\tilde{f}_n(\vx_n(t),\xi_{n,t})]^\top$}.\\ Similarly, one can define \mbox{$n{\times} d$} matrices $\mU(t)$, $\mZ(t)$, as well as $\vy(t)$ as the vector of $y_i(t)$. Therefore, Algorithm~\ref{alg:comp-push-sum} with Option  \textcolor{blue}{I} (average consensus) may be written as follows:
	\begingroup
	\allowdisplaybreaks
	\begin{subequations}\label{eq:update-cons}
		\begin{align}
			\hat{\mX}{(t{+}1)} &\coloneqq \hat{\mX}{(t)} + Q(\mX{(t)}-\hat{\mX}{(t)}),\label{eq:update-cons-1}\\
			\mX{(t{+}1)} &\coloneqq \mX{(t)} + \gamma\left(\mW{-}\mI\right)\hat{\mX}{(t{+}1)},\label{eq:update-lazy-cons}\\
			\vy(t{+}1)&\coloneqq \mW \vy(t),\\
			\vz_i(t{+}1)&\coloneqq {\vx_i(t{+}1)}/{y_i(t{+}1)},
		\end{align}
	\end{subequations}
	\endgroup
	where, due to the properties mentioned in~\eqref{eq:mixing-matrix}, it holds that $\overline{\mX}(t) = \overline{\mX}(0)$, for all {$t\geq 0$}. Furthermore, Algorithm~\ref{alg:comp-push-sum} with Option \textcolor{blue}{II} can be written in matrix notation as
	\begingroup
	\allowdisplaybreaks
	\begin{subequations}\label{eq:update-opt}
		\begin{align}
			\hat{\mX}{(t{+}1)} &\coloneqq \hat{\mX}{(t)} + Q(\mX{(t)}-\hat{\mX}{(t)}),\\
			\mU{(t{+}1)} &\coloneqq \mX{(t)} + \gamma\left(\mW{-}\mI\right)\hat{\mX}{(t{+}1)},\label{eq:update-lazy-opt}\\
			\vy(t{+}1)&\coloneqq \mW \vy(t),\\
			\vz_i(t{+}1)&\coloneqq {\vu_i(t{+}1)}/{y_i(t{+}1)},\\
			\mX{(t{+}1)} &\coloneqq \mU{(t{+}1)} - \eta\, \partial\tilde{F}(\mZ{(t{+}1)},\vxi_{t{+}1}),
		\end{align}
	\end{subequations}
	\endgroup
	where \mbox{$\eta>0$} indicates the \textit{optimization stepsize}, which  we consider to be constant. Note that while the consensus problem is an optimization algorithm, the proposed method does not achieve linear convergence rate for the stochastic optimization problem formulation. Thus, we studied the two problem classes, consensus and stochastic optimization, independently.

	\begin{table}[!t]
		\caption{Comparison of the convergence rates and valid compression ratios for \textbf{decentralized consensus} algorithms.}
		\label{tab:cons-comparison}
		\vspace{-0.7em}
		\begin{center}
			\begin{minipage}{\linewidth}
				\centering
				\resizebox{\linewidth}{!}{
					\begin{tabular}{lclcl} \toprule Algorithm & {\small Directed}\footnote{Undirected or directed networks.} & Linear Rate\footnote{The constant in the linear convergence rates.} & valid $\omega$ & $\gamma$\\[1pt]
						
						\midrule
						\midrule
						
						\footnotesize\textbf{Xiao \& Boyd}~\cite{xiao2004fast} & \xmark & $\mcO\big(1{-}\delta\big)$ & no compression & N/A\\[1pt]
						\midrule
						
						\footnotesize\textbf{Koloskova et al.}~\cite{koloskova2019decentralized} & \xmark & $\mcO\big(1{-}\delta^2\omega\big)$ & $(0,\,1]$ & $\mcO(\omega\delta^2)$\\[1pt]
						\midrule
						
						\footnotesize\textbf{Zhang et al.}~\cite{zhang2021innovation} & \xmark
						& $\mcO\big(1{-}\delta\omega\big)$ & $(0,\,1]$ & $\mcO(\omega)$\\[1pt]
						\midrule
						
						\footnotesize\textbf{Kempe et al.}~\cite{kempe2003gossip}& \cmark
						& $\mcO\big(1{-}\delta\big)$ & no compression & N/A\\[1pt]
						\midrule

						\footnotesize\textbf{Taheri et al.}~\cite{taheri2020quantized} & \cmark
						& $\mcO\big(1{-}\delta\big)$ & $\Big[\Theta \Big(\frac{\delta^{-2}}{1{+}\delta^{-2}}\Big),1\Big]$\footnote{This only shows the reliance on the spectral gap $\delta$, while $C,\beta$ are skipped. $\Theta$ indicates the same asymptotic upper and lower bounds, while $\mcO$ indicates only the asymptotic upper bound.} & N/A\\[1pt]

						\midrule
						\footnotesize\textbf{\color{magenta} This Work}& \cmark & $\mcO\big(1{-}\delta^2\omega^2\big)$ & $(0,\,1]$ & $\mcO(\omega\delta)$\\[1pt]
						\bottomrule[1pt]
					\end{tabular}
				}
			\end{minipage}
		\end{center}
		
	\end{table}

	\noindent\textbf{$\diamond$ Assumptions \& Highlights of the Results:}
	The decentralized consensus algorithm in~\eqref{eq:update-cons} is similar to CHOCO-Gossip~\cite{koloskova2019decentralized} with an additional slack parameter $y_i$ for push-sum. Table~\ref{tab:cons-comparison} provides a comparison of the linear rates and compression intervals for different consensus methods. As shown in the table, with a suboptimal choice of $\gamma$, we show a linear convergence rate with an arbitrary compression ratio. The same behavior was obtained by~\cite{koloskova2019decentralized} only for undirected graphs. Moreover, note that Zhang et al.~\cite{zhang2021innovation} presented an optimal convergence rate for CHOCO-Gossip. We conjecture that our algorithm's convergence rate dependence on $\delta$ can be improved accordingly. This will be left for future work.
	
	In addition to the consensus result, which is the baseline of our work, we further provide the analysis of~\eqref{eq:update-opt} for
	the following three function classes:
	\begin{enumerate}[leftmargin=3.5em,label=(\roman*)]
		\item smooth and strongly convex,
		\item smooth and convex,
		\item smooth and non-convex,
	\end{enumerate}
	where for each function class, we consider a subset of the following assumptions.
	\begin{assumption}[Bounded Variance]\label{assump:bounded-variance}
		Stochastic gradients have bounded variance, i.e., for all $i{\in}[n]$,
		\begin{align*}
			\bbE_{\xi_i\sim\mcD_i} \lVert \nabla\tilde{f}_i(\vx,\xi_i) - \nabla f_i(\vx)\rVert^2 \leq \sigma^2.
		\end{align*}
	\end{assumption}
	
	\begin{assumption}[Bounded Gradients]\label{assump:bounded-gradient}
		There exists a constant $G$ that for all $\vx\in\bbR^d$, each local gradient $\nabla\tilde{f}_i(\vx,\xi_i)$ has bounded second moment, i.e., for all $i\in[n]$,
		\begin{align*}
			\bbE_{\xi_i\sim\mcD_i} \lVert \nabla\tilde{f}_i(\vx,\xi_i) \rVert^2 \leq G^2.
		\end{align*}
	\end{assumption}
	
	\begin{assumption}[Smooth Gradients]\label{assump:l-smooth}
		Each function $f_i(\vx)$, for all $i\in[n]$ is $L$-smooth, i.e., for all $\vx,\vy\in\bbR^d$,
		\begin{align*}
			\lVert\nabla f_i(\vx) - \nabla f_i(\vy)\rVert\leq L \lVert \vx - \vy \rVert.
		\end{align*}
	\end{assumption}
	
	\begin{assumption}[Convexity]\label{assump:convex}
		Each function $f_i(\vx)$, for all $i\in[n]$ is convex, i.e., for all $\vx,\vy\in\bbR^d$,
		\begin{align*}
			f_i(\vx) + \left\langle \nabla f_i(\vx),\vy-\vx\right\rangle \leq f_i(\vy),
		\end{align*}
		and $\vx^\star \coloneqq \argmin\limits_{\vx\in\bbR^d} f(\vx)$.
	\end{assumption}
	
	\begin{assumption}[Strong Convexity]\label{assump:strong-convex}
		Each function $f_i(\vx)$, for all $i\in[n]$ is $\mu$-strongly convex, i.e., for all $\vx,\vy\in\bbR^d$,
		\begin{align*}
			f_i(\vx) + \left\langle \nabla f_i(\vx),\vy-\vx\right\rangle + \frac{\mu}{2} \lVert\vy-\vx\rVert^2 \leq f_i(\vy).
		\end{align*}
	\end{assumption}

	\begin{table}[!t]
		\caption{Comparison of the convergence rates and compression ratios for \textbf{smooth \& convex} stochastic optimization.}
		\label{tab:convex-comparison}
		\vspace{-0.7em}
		\begin{center}
			\begin{minipage}{\linewidth}
				\centering
				\resizebox{\linewidth}{!}{
					\begin{tabular}{lcclc} \toprule
						Algorithm & {\small Directed} & S.C.\footnote{S.C. denotes strongly convex functions.} & Rate\footnote{Sublinear Convergence Rate to an optimal solution.} & valid $\omega$ \\
						
						\midrule
						\midrule
						
						\footnotesize\textbf{Nedi\'{c} \& Olshevsky}~\cite{nedic2016stochastic} & \cmark & \cmark & $\mcO\big(\frac{\log T}{nT}\big)$ & no compression\\[1pt]
						\midrule
						
						\footnotesize\textbf{Koloskova et al.}~\cite{koloskova2019decentralized}\footnote{There is no analysis for smooth convex (not S.C.) functions in~\cite{koloskova2019decentralized}, while we obtain this rate accordingly.} & \xmark & \xmark & $\mcO\big(\frac{1}{\sqrt{nT}}\big)$ & $(0,\,1]$\\[1pt]
						\midrule
						
						\footnotesize\textbf{Koloskova et al.}~\cite{koloskova2019decentralized} & \xmark & \cmark & $\mcO\big(\frac{1}{nT}\big)$ & $(0,\,1]$\\[1pt]
						\midrule
						
						\footnotesize\textbf{Taheri et al.}~\cite{taheri2020quantized} & \cmark & \xmark
						& $\mcO\big(\frac{1}{\sqrt{nT}}\big)$ & $\Big[\Theta \Big(\frac{\delta^{-2}}{1{+}\delta^{-2}}\Big),1\Big]$\\[1pt]
						\midrule
						
						\footnotesize\textbf{\color{magenta}This Work} & \cmark & \xmark & $\mcO\big(\frac{1}{\sqrt{nT}}\big)$ & $(0,\,1]$\\[1pt]
						\midrule
						
						\footnotesize\textbf{\color{magenta}This Work} & \cmark & \cmark & $\mcO\big(\frac{\log T}{nT}\big)$ & $(0,\,1]$\\[1pt]
						\bottomrule[1pt]
					\end{tabular}
				}
			\end{minipage}
		\end{center}
	\end{table}
	
	These assumptions are common in the literature, and we refer to them for our analysis.
	We consider Assumptions~\ref{assump:bounded-variance}-\ref{assump:l-smooth} for all the three function classes. Assumption~\ref{assump:bounded-gradient} holds for a subset of classification problems such as overparametrized neural network models with Sigmoid activation functions~\cite{rasamoelina2020review} and constrained optimization problems. The relaxation of such assumption is for undirected communication networks is studied in~\cite{kovalev2021linearly}, and the extension to directed graphs remains an open question. Table~\ref{tab:convex-comparison} compares the convergence rate of~\eqref{eq:update-opt} with~\cite{nedic2016stochastic,koloskova2019decentralized,taheri2020quantized}. Note that~\eqref{eq:update-opt} obtains the same convergence rate $\mcO({1}/{\sqrt{nT}})$ as~\cite{taheri2020quantized} for smooth and convex problems, while no restrictions on the compression ratio. Besides~\cite{koloskova2019decentralized} which only considers the strong convexity assumption, we also analyze CHOCO under convex assumption. Table~\ref{tab:convex-comparison} also shows the convergence properties of~\eqref{eq:update-opt} for smooth and strongly convex problems. Our analysis shows the same convergence rate $\tilde{\mcO}({1}/{nT})$ as CHOCO-SGD (for undirected networks), up to a logarithmic factor. Note that in~\cite{koloskova2019decentralized}, the authors consider a decreasing stepsize $\{\eta_t\}_{t\geq 0}$. However, we consider a fixed stepsize $\eta$ in our analysis. To the best of our knowledge, this is the first results on compressed push-sum for stochastic optimization under strong convexity and smoothness assumptions.
	
	Finally, we consider the class of smooth and non-convex objectives. Table~\ref{tab:nonconvex-comparison} shows the same sublinear convergence rates to reach first-order stationary points for existing algorithms. In this scenario, our algorithm enables an arbitrary compression ratio compared to~\cite{taheri2020quantized} and extends the results in~\cite{koloskova2019decentralized2} to directed networks.
	
	\begin{table}[!t]
		\caption{Comparison of the sublinear convergence rates to first-order stationary points and valid compression ratios for \textbf{smooth \& non-convex} stochastic optimization.}
		\label{tab:nonconvex-comparison}
		\vspace{-0.7em}
		\begin{center}
			\begin{minipage}{\linewidth}
				\centering
				\resizebox{\linewidth}{!}{
					\begin{tabular}{lccc} \toprule
						Algorithm & {\small Directed} & Rate & valid $\omega$\\
						
						\midrule
						\midrule
						
						\footnotesize\textbf{Assran et al.}~\cite{assran2019stochastic} & \cmark & $\mcO\big(\frac{1}{\sqrt{nT}}\big)$ & no compression\\[1pt]
						\midrule
						
						\footnotesize\textbf{Koloskova et al.}~\cite{koloskova2019decentralized2} & \xmark & $\mcO\big(\frac{1}{\sqrt{nT}}\big)$ & $(0,\,1]$\\[1pt]
						\midrule
						
						\footnotesize\textbf{Taheri et al.}~\cite{taheri2020quantized} & \cmark
						& $\mcO\big(\frac{1}{\sqrt{nT}}\big)$ & $\Big[\Theta \Big(\frac{\delta^{-2}}{1{+}\delta^{-2}}\Big),1\Big]$\\[1pt]
						
						\midrule
						\footnotesize\textbf{\color{magenta}This Work} & \cmark & $\mcO\big(\frac{1}{\sqrt{nT}}\big)$ & $(0,\,1]$\\[1pt]
						\bottomrule[1pt]
					\end{tabular}
				}
			\end{minipage}
		\end{center}
	\end{table}


	\section{Convergence Analysis}
	In this section, we analyze the convergence properties of the proposed algorithm. Before stating the results, we first present a proposition here.
	\label{sec:analysis}
	\begin{proposition}\label{prop:lazy-mixing}
		Let $\gamma\in(0,1]$, and \mbox{$\mB=(1{-}\gamma)\mI+\gamma\mW$}. Then, for column stochastic matrix $\mB$, the following hold:
		\begin{align}\label{eq:lazy-matrix}
			\mB\vphi=\vphi,\,\,\, [\mB^t \vect{1}]_i \geq \kappa, \,\,\, \big\lVert\mB^t{-}\vphi\vect{1}^\top\big\rVert\leq C(1{-}\gamma\delta)^t.
		\end{align}
	\end{proposition}
	
	\begin{proof}[Proof of Proposition~\ref{prop:lazy-mixing}]
		First, note that:
		\begin{align}\label{eq:prop-proof-1}
			\mB\vphi = \left[(1{-}\gamma)\mI+\gamma\mW\right]\vphi \overset{\text{\footnotesize\eqref{eq:mixing-matrix}}}{=} (1{-}\gamma)\vphi+\gamma\vphi = \vphi,
		\end{align}
		Moreover, we have
		\begin{align}\label{eq:prop-proof-2}
			\left[\mB\vect{1}\right]_i &= \left[\sum\limits_{s=0}^{t}\binom{t}{s}(1{-}\gamma)^{t{-}s}\gamma^s\mW^s\vect{1}\right]_i\nonumber\\
			&=\sum\limits_{s=0}^{t}\binom{t}{s}(1{-}\gamma)^{t{-}s}\gamma^s\left[\mW^s\vect{1}\right]_i\nonumber\\
			\overset{\text{\footnotesize\eqref{eq:mixing-matrix}}}&{\geq} \sum\limits_{s=0}^{t}\binom{t}{s}(1{-}\gamma)^{t{-}s}\gamma^s\kappa\nonumber\\
			&= \kappa.
		\end{align}
		Finally, we can see that:
		\begin{align}\label{eq:prop-proof-3}
			\left\lVert\mB^t{-}\vphi\vect{1}^\top\right\rVert &= \left\lVert\sum\limits_{s=0}^{t}\binom{t}{s}(1{-}\gamma)^{t{-}s}\gamma^s\mW^s{-}\vphi\vect{1}^\top\right\rVert\nonumber\\
			&= \left\lVert\sum\limits_{s=0}^{t}\binom{t}{s}(1{-}\gamma)^{t{-}s}\gamma^s\left[\mW^s{-}\vphi\vect{1}^\top\right]\right\rVert\nonumber\\
			\overset{\text{\footnotesize tri. ineq.}}&{\leq} \sum\limits_{s=0}^{t}\binom{t}{s}(1{-}\gamma)^{t{-}s}\gamma^s \left\lVert\mW^s{-}\vphi\vect{1}^\top\right\rVert\nonumber\\
			\overset{\eqref{eq:mixing-matrix}}&{\leq} C\sum\limits_{s=0}^{t}\binom{t}{s}(1{-}\gamma)^{t{-}s}\gamma^s (1-\delta)^s\nonumber\\
			&=C(1{-}\gamma\delta)^t.
		\end{align}
		Hence,~\eqref{eq:prop-proof-1},~\eqref{eq:prop-proof-2}, and~\eqref{eq:prop-proof-3} conclude the proof of Proposition~\ref{prop:lazy-mixing}.
		\vspace{-1em}
	\end{proof}

	Proposition~\ref{prop:lazy-mixing} indicates that similar properties as~\eqref{eq:mixing-matrix} hold for matrix $\mB$ with a contracted spectral gap $\gamma\delta$. Note that both~\eqref{eq:update-lazy-cons} and~\eqref{eq:update-lazy-opt} contain update rules with consensus stepsize $\gamma$. This can be written as
	\begin{align}\label{eq:update-lazy}
		\mX(t) + \gamma\,(\mW&{-}\mI)\hat{\mX}(t{+}1)\nonumber\\
		&= \mB \mX(t) + \gamma\, (\mW{-}\mI)(\hat{\mX}(t{+1})-\mX(t)),
	\end{align}
	which implies a consensus using mixing matrix $\mB$ with feedback. Proposition~\ref{prop:lazy-mixing} will be used in the analysis of the next theorems. Now, we present our convergence result for the consensus problem.
	
	\begin{theorem}[Consensus]\label{thm:consensus}
		Let the compression operator $Q$ satisfy~\eqref{eq:q-comp} with \mbox{$\omega\in(0,1]$}, \mbox{$\hat{\mX}{(0)}= \vect{0}$}, and \mbox{$\vy{(0)}= \vect{1}$}. Then, the iterates of
		update rule~\eqref{eq:update-cons}
		have the following property:
		\begin{align*}
			\bbE\Psi_{z}(t) \leq C_0\rho^t,
		\end{align*}
		where \mbox{$\Psi_{z}(t)\coloneqq\big\lVert \mZ(t) {-} \overline{\mX}(0)\big\rVert_F$}, \mbox{$\rho\coloneqq 1{-}\frac{\omega^2\delta^2}{16\beta\delta+8\beta^2C+8\omega\delta^2}$},\\
		and $C_0\coloneqq \frac{4nC(1{+}\beta C) \lVert\mX(0)\rVert_F}{\kappa\delta}$, when \mbox{$\gamma{\coloneqq}\frac{2\omega\delta}{8\beta\delta+4\beta^2C+4\omega\delta^2}$}.
	\end{theorem}
	
	\begin{proof}[Proof of Theorem~\ref{thm:consensus}]
		Based on the update rule in~\eqref{eq:update-cons}, consider the two error functions \mbox{$\mcL(t)=\bbE\lVert\mX(t){-}\vphi\vect{1}^\top\mX(0)\rVert_F$} and \mbox{$\hat{\mcL}(t)=\bbE\lVert\hat\mX(t{+}1){-}\mX(t)\rVert_F$}. By unraveling the recursion in~\eqref{eq:update-lazy}, we have
		\begin{align}\label{eq:cons-proof-1}
			\mX(t{+}1) \overset{\eqref{eq:update-lazy}}&{=} \mB\mX(t) + \gamma\,(\mW{-}\mI)\left[\hat\mX(t{+}1)-\mX(t)\right]\nonumber\\
			& \,\,\, \vdots \nonumber\\
			&= \mB^{t{+}1} \mX(0)\nonumber\\
			& + \gamma\,\sum\limits_{s=0}^{t} \mB^s\left(\mW{-}\mI\right)\left[\hat\mX(t{-}s{+}1) - \mX(t{-}s)\right].
		\end{align}
		Then, we use~\eqref{eq:cons-proof-1} to provide an upper bound on \mbox{$\mcL(t{+}1)$}:
		\begin{align}\label{eq:cons-proof-2}
			\mcL&(t{+}1) = \bbE\left\lVert\mX(t{+}1) - \vphi\vect{1}^\top\mX(0)\right\rVert_F\nonumber\\
			\overset{\eqref{eq:cons-proof-1}}&{=} \bbE\Bigg\lVert \mB^{t{+}1} \mX(0) - \vphi\vect{1}^\top\mX(0)\nonumber\\
			&+ \gamma\,\sum\limits_{s=0}^{t} \mB^s\left(\mW{-}\mI\right)\left[\hat\mX(t{-}s{+}1) - \mX(t{-}s)\right]\Bigg\rVert_F\nonumber\\
			\overset{\eqref{eq:mixing-matrix}}&{=} \bbE\Bigg\lVert \left(\mB^{t{+}1} - \vphi\vect{1}^\top\right)\mX(0)\nonumber\\
			&+ \gamma\,\sum\limits_{s=0}^{t} \left(\mB^s{-}\vphi\vect{1}^\top\right)\left(\mW{-}\mI\right)\left[\hat\mX(t{-}s{+}1) - \mX(t{-}s)\right]\Bigg\rVert_F\nonumber\\
			\overset{\text{\footnotesize tri. ineq.}}&{\leq} \left\lVert \left(\mB^{t{+}1} - \vphi\vect{1}^\top\right)\mX(0)\right\rVert_F\nonumber\\
			&+ \gamma\sum\limits_{s=0}^{t}\bbE\left\lVert \left(\mB^s{-}\vphi\vect{1}^\top\right)\left(\mW{-}\mI\right)\left[\hat\mX(t{-}s{+}1) {-} \mX(t{-}s)\right]\right\rVert_F\nonumber\\
			\overset{\eqref{eq:lazy-matrix}}&{\leq} C{(1{-}\gamma\delta)}^{t{+}1}\left\lVert\mX(0)\right\rVert_F\nonumber\\
			&+ \gamma\,C\beta\sum\limits_{s=0}^{t}\left(1-\gamma\delta\right)^s\bbE\left\lVert \hat\mX(t{-}s{+}1) - \mX(t{-}s)\right\rVert_F.
		\end{align}
		Before stating a bound on \mbox{$\hat{\mcL}(t{+}1)$}, note that Jensen's inequality implies the following property
		\begin{align}\label{eq:q-comp-power-1}
			\bbE\big\lVert Q(\vx) - \vx \big\rVert \leq \sqrt{1{-}\omega}\, \big\lVert \vx \big\rVert, \qquad \forall \vx\in\bbR^d,
		\end{align}
		as an immediate result of~\eqref{eq:q-comp}. Therefore, we have
		\begin{align}\label{eq:cons-proof-3}
			\hat{\mcL}(t{+}1)&=\bbE\Big\lVert\mX(t{+}1) - \hat{\mX}(t{+}2)\Big\rVert_F\nonumber\\
			\overset{\eqref{eq:update-cons-1},\eqref{eq:q-comp-power-1}}&{\leq} \sqrt{1-\omega}\,\left\lVert\mX(t{+}1) - \hat{\mX}(t{+}1)\right\rVert_F,
		\end{align}
		where according to the definition of $\mX(t{+}1)$ in~\eqref{eq:update-lazy-cons},
		\begin{align}\label{eq:cons-proof-4}
			\bbE\Big\lVert\mX(t{+}1) &- \hat{\mX}(t{+}1)\Big\rVert_F\nonumber\\
			\overset{\eqref{eq:update-lazy-cons}}&{=} \bbE \left\lVert\mX(t) + \gamma\left(\mW-\mI\right)\hat{\mX}(t{+}1) - \hat{\mX}(t{+}1)\right\rVert_F\nonumber\\
			\overset{}&{=} \bbE\Big\lVert \gamma\left(\mW-\mI\right)\mX(t)\nonumber\\
			&+ \left((1{+}\gamma)\mI - \gamma\mW\right) \left[\mX(t)-\hat{\mX}(t{+}1)\right]\Big\rVert_F\nonumber\\
			\overset{\eqref{eq:mixing-matrix}}&{=} \bbE\Big\lVert \gamma\left(\mW-\mI\right)\left[\mX(t)-\vphi\vect{1}^\top\mX(0)\right]\nonumber\\
			&+ \left((1{+}\gamma)\mI - \gamma\mW\right) \left[\mX(t)-\hat{\mX}(t{+}1)\right]\Big\rVert_F\nonumber\\
			\overset{\text{\footnotesize tri. ineq.}}&{\leq} \gamma\beta\,\bbE\Big\lVert\mX(t)-\vphi\vect{1}^\top\mX(0)\Big\rVert_F\nonumber\\
			&+ (1{+}\gamma\beta)\,\bbE\Big\lVert \mX(t)-\hat{\mX}(t{+}1)\Big\rVert_F.
		\end{align}
		Therefore, according to~\eqref{eq:cons-proof-2},~\eqref{eq:cons-proof-3}, and~\eqref{eq:cons-proof-4}, we have
		\begingroup
		\allowdisplaybreaks
		\begin{subequations}\label{eq:cons-inequalities}
			\begin{align}
				\mcL(t{+}1) &\leq C(1{-}\gamma\delta)^{t{+}1} \lVert \mX(0)\rVert_F {+} \gamma C\beta\sum\limits_{s{=}0}^t (1{-}\gamma\delta)^s \hat{\mcL}(t{-}s),\label{eq:cons-inequalities-a}\\
				\hat{\mcL}(t{+}1) &\leq  \gamma\beta\tilde{\omega}\,\mcL(t) + (1{+}\gamma\beta)\tilde{\omega}\,\hat{\mcL}(t),\label{eq:cons-inequalities-b}
			\end{align}
		\end{subequations}
		\endgroup
		where \mbox{$\tilde{\omega}=\sqrt{1-\omega}$}.
		Given the two inequalities in~\eqref{eq:cons-inequalities}, we use induction to show that for any \mbox{$\omega\in(0,1]$}, under some suboptimal choices of $\rho$ and $\gamma$, as introduced in Theorem~\ref{thm:consensus}, the following inequality holds:
		%
		\begin{align}\label{eq:cons-L-hat-inequality}
			\hat{\mcL}(t) \leq \frac{C\lVert\mX(0)\rVert_F}{1{-}\gamma\delta} \rho^t \coloneqq a_1\rho^t.
		\end{align}
		Note that $\hat{\mX}(0)=\vect{0}$, thus according to~\eqref{eq:q-comp-power-1}, the inequality in~\eqref{eq:cons-L-hat-inequality} holds for the base, $t=0$. Now, let us assume that~\eqref{eq:cons-L-hat-inequality} holds for $t=0,1,2,\dots,T$. Then, replacing~\eqref{eq:cons-inequalities-a} in~\eqref{eq:cons-inequalities-b}, and using the induction's assumption, we obtain
		\begin{align}\label{eq:cons-proof-5}
			\hat{\mcL}(T{+}1) &\leq (1{+}\gamma\beta)\tilde{\omega}\,\hat{\mcL}(T) + \gamma\beta\tilde{\omega}\Bigg[C(1{-}\gamma\delta)^{T} \lVert \mX(0)\rVert_F\nonumber\\
			& \qquad\qquad\qquad\quad+ \gamma C\beta\sum\limits_{s{=}0}^{T{-}1} (1{-}\gamma\delta)^s \hat{\mcL}(T{-}s{-}1)\Bigg]\nonumber\\
			&\leq a_1(1{+}\gamma\beta)\tilde{\omega}\rho^T + \frac{\gamma\beta\tilde{\omega}C \lVert\mX(0)\rVert_F}{1{-}\gamma\delta}(1{-}\gamma\delta)^{T{+}1}\nonumber\\
			& \qquad\qquad\qquad + a_1\gamma^2\beta^2\tilde{\omega}\,C\rho^{T{-}1}\sum\limits_{s{=}0}^{T{-}1} \left(\frac{1{-}\gamma\delta}{\rho}\right)^s\nonumber\\
			&\leq \frac{a_1(1{+}\gamma\beta)\tilde{\omega}}{\rho}\rho^{T{+}1} + a_1\gamma\beta\tilde{\omega}(1{-}\gamma\delta)^{T{+}1}\nonumber\\
			& \qquad\qquad\qquad + \frac{a_1\gamma^2\beta^2\tilde{\omega}\,C\rho^{T{-}1}}{\rho^2\left(1-\frac{1{-}\gamma\delta}{\rho}\right)}\rho^{T{+}1}\nonumber\\
			\overset{}&{\leq} a_1\tilde{\omega}\Bigg[\frac{1{+}\gamma\beta}{\rho} + \gamma\beta + \frac{\gamma^2\beta^2C}{\rho(\rho{-}1{+}\gamma\delta)}\Bigg]\rho^{T{+}1},
		\end{align}
		where the last inequality holds due to \mbox{$1{-}\gamma\delta<\rho$}. To prove the inequality in~\eqref{eq:cons-L-hat-inequality} for \mbox{$t=T{+}1$}, it is enough to show that the upper bound in~\eqref{eq:cons-proof-5} is bounded by \mbox{$a_1\rho^{T{+}1}$}.
		In other words, it is sufficient to check that 
		under the choices of $\gamma$ and $\rho$ as in Theorem~\ref{thm:consensus}, the following inequality holds:
		\begin{align}\label{eq:cons-constant-1}
			\sqrt{1{-}\omega}\left(\frac{1{+}\gamma\beta}{\rho} + \gamma\beta + \frac{\gamma^2\beta^2C}{\rho\,(\rho{-}1{+}\gamma\delta)}\right)\leq 1,
		\end{align}
		for any compression ratio \mbox{$\omega\in(0,1]$}. On the one hand, by definition, we have: \begin{align}\label{eq:cons-constant-2}
			\gamma = \frac{2\omega\delta}{8\beta\delta + 4\beta^2C+4\omega\delta^2} \leq \frac{2\omega\delta}{8\beta\tilde{\omega}\delta + 4\tilde{\omega}\beta^2C+\omega\delta^2}&\nonumber\\
			= \frac{2(1-\tilde{\omega}^2)\delta}{4\beta\tilde{\omega}(2\delta + \beta C)+(1-\tilde{\omega}^2)\delta^2} &\Rightarrow\nonumber\\
			4\gamma\beta\tilde{\omega}(2\delta+\beta C) \leq (1-\tilde{\omega}^2) (2\delta - \gamma\delta^2)&\Rightarrow\nonumber\\
			\gamma^2\beta\tilde{\omega}(2\delta+\beta C) \leq (1-\tilde{\omega}-\epsilon) (\gamma\delta - \epsilon)&,
		\end{align}
		where $\epsilon = \gamma\delta(1-\tilde{\omega})/2$.
		On the other hand
		\begin{align}\label{eq:cons-constant-3}
			\gamma^2\beta\tilde{\omega}(2\delta+\beta C) &= \gamma\beta\tilde{\omega}(2\gamma\delta+\gamma\beta C) \nonumber\\
			&\geq \gamma\beta\tilde{\omega}((2-\epsilon)(\gamma\delta-\epsilon)+\gamma\beta C).
		\end{align}
		Therefore, according to~\eqref{eq:cons-constant-2} and~\eqref{eq:cons-constant-3}, we have:
		\begin{align}\label{eq:cons-constant-4}
			\gamma\beta\tilde{\omega}((2-\epsilon)(\gamma\delta-\epsilon)+\gamma\beta C) \leq (1-\tilde{\omega}-\epsilon) (\gamma\delta - \epsilon) &\Rightarrow \nonumber\\
			\gamma\beta\tilde{\omega}\left(\frac{2-\epsilon}{1-\epsilon}+\frac{\gamma\beta C}{(1-\epsilon)(\gamma\delta-\epsilon)}\right) \leq \frac{1-\tilde{\omega}-\epsilon}{1-\epsilon}&\Rightarrow \nonumber\\
			\tilde{\omega}\left(\frac{\gamma\beta}{1-\epsilon}+\gamma\beta+\frac{\gamma^2\beta^2 C}{(1-\epsilon)(\gamma\delta-\epsilon)}\right) \leq 1 - \frac{\tilde{\omega}}{1-\epsilon}&\Rightarrow \nonumber\\
			\tilde{\omega}\left(\frac{1+\gamma\beta}{1-\epsilon}+\gamma\beta+\frac{\gamma^2\beta^2 C}{(1-\epsilon)(\gamma\delta-\epsilon)}\right) \leq 1 &\Rightarrow \nonumber\\
			\tilde{\omega}\left(\frac{1+\gamma\beta}{1-\tilde{\epsilon}}+\gamma\beta+\frac{\gamma^2\beta^2 C}{(1-\tilde{\epsilon})(\gamma\delta-\tilde{\epsilon})}\right) \leq 1 &,
		\end{align}
		for any $\epsilon$ and $\tilde{\epsilon}$ such that \mbox{$0<\tilde{\epsilon}\leq \epsilon < \gamma\delta$}. Given the fact that $1{+}\tilde{\omega}\leq 2$, we set
		\begin{align}\label{eq:cons-constant-5}
			\tilde{\epsilon} = \frac{\gamma\delta(1-\tilde{\omega}^2)}{4} \leq  \frac{\gamma\delta(1-\tilde{\omega})}{2} = \epsilon,
		\end{align}
		therefore, we have \mbox{$\rho=1{-}\tilde{\epsilon}$}, which turns~\eqref{eq:cons-constant-4} into~\eqref{eq:cons-constant-1}. This implies that~\eqref{eq:cons-L-hat-inequality} also holds for \mbox{$t=T{+}1$}.
		Hence, by induction principle,~\eqref{eq:cons-L-hat-inequality} holds for all \mbox{$t\in\mcZ_{0}^{+}$}.
		We emphasize that \mbox{$\gamma=\mcO(\omega\delta)$} is a suboptimal (conservative) choice for the consensus stepsize that guarantees inequality~\eqref{eq:cons-L-hat-inequality} for any compression ratio \mbox{$\omega\in(0,1]$}. This may be relaxed similar to~\cite{zhang2021innovation}. Moreover, by replacing~\eqref{eq:cons-L-hat-inequality} in~\eqref{eq:cons-inequalities-a}, we have
		\begin{align}\label{eq:cons-L-inequality}
			\mcL(t) &\leq C (1{-}\gamma\delta)^{t} \lVert\mX(0)\rVert_F + \gamma\beta\,C \sum\limits_{s{=}0}^{t{-}1}(1{-}\gamma\delta)^s a_1 \rho^{t{-}s{-}1} \nonumber\\
			&\leq \left(C \lVert\mX(0)\rVert_F + \frac{\gamma\beta\,C a_1}{\rho-1+\gamma\delta}\right)\rho^t \nonumber\\
			&= \left(C \lVert\mX(0)\rVert_F + \frac{\gamma\beta\,C a_1}{\gamma\delta-\frac{\gamma\delta\omega}{4}}\right)\rho^t \nonumber\\
			&= C \lVert\mX(0)\rVert_F \left(1 + \frac{\beta\,C}{\delta\left(1-\tfrac{\omega}{4}\right)(1-\gamma\delta)}\right)\rho^t \nonumber\\
			&\leq \frac{C(1+\beta\,C)\lVert\mX(0)\rVert_F}{\frac{3}{4}\times\frac{\delta}{2}} \rho^t \nonumber\\
			&= \frac{8C(1{+}\beta C) \lVert\mX(0)\rVert_F}{3\delta} \rho^t,
		\end{align}
		where we used the fact that $\gamma\leq\frac{1}{2\delta}$ in the last inequality. Finally, using a standard technique for push-sum analysis~\cite{tsianos2012push,nedic2016stochastic,nedic2015nonasymptotic,assran2019stochastic}, we have
		\begin{align}\label{eq:cons-Z-inequality}
			\bbE \lVert \vz_i(t) {-} \overline{\vx}(0)\rVert & = \bbE \left\lVert \frac{\vx_i(t)}{y_i(t)} - \overline{\vx}(0)\right\rVert \nonumber\\
			& = \bbE \left\lVert \frac{\vx_i(t)}{[(\mW^t-\vphi\vect{1}^\top)\vect{1}]_i + [\vphi]_i n} - \overline{\vx}(0)\right\rVert \nonumber\\
			\overset{\text{\footnotesize tri. ineq.}}&{\leq} \frac{1}{\kappa}\bbE\left\lVert\vx_i(t) {-}[\vphi]_i n\overline{\vx}(0)\right\rVert\nonumber\\
			& + \frac{1}{n\kappa}\left\lVert[(\mW^t-\vphi\vect{1}^\top)\vect{1}]_i\right\rVert\left\lVert \mX(0)^\top \vect{1}\right\rVert \nonumber\\
			\overset{\eqref{eq:mixing-matrix},\eqref{eq:cons-L-inequality}}&{\leq} \frac{4C(1{+}\beta C) \lVert\mX(0)\rVert_F}{\kappa\delta}\rho^t,
		\end{align}
		for all \mbox{$i\in[n]$}. Note that in the last inequality of~\eqref{eq:cons-Z-inequality}, we also considered the fact that \mbox{$1{-}\delta\leq\rho$}. Moreover, we have
		\begin{align}
			\bbE\big\lVert \mZ(t) {-} \overline{\mX}(0)\big\rVert_F \leq \sum\limits_{i=1}^n \bbE \lVert \vz_i(t) {-} \overline{\vx}(0)\rVert
		\end{align}
		which concludes the statement of Theorem~\ref{thm:consensus}.
	\end{proof}
	
	Theorem~\ref{thm:consensus} guarantees a linear convergence of the rescaled parameters $\vz_i(t)$ to the average parameter $\overline{\vx}(0)$, under some proper consensus stepsize $\gamma$, e.g., as in Theorem~\ref{thm:consensus}. The rate depends quadratically on $\omega\delta$. We conjecture that similar to~\cite{zhang2021innovation}, this dependence is also linear~Table~\ref{tab:cons-comparison}. For example, directed regular graphs have $\delta$ with a cubic worst case dependence on the number of agents $n$~\cite{olshevsky2006convergence}. In such case, the linear convergence rate has dependence $\mcO(n^6)$ on the number of agents. Check~\cite{coste2021spectral} for more details on the spectral gap of random digraphs.
	
	We now proceed to present the stochastic optimization results. Consider the update rule in~\eqref{eq:update-opt}. First, we present a technical lemma that helps prove the convergence theorems for the three mentioned function classes.

	\begin{lemma}\label{lem:diff-opt} 
		Let the compression operator $Q$ satisfy~\eqref{eq:q-comp} with \mbox{$\omega\in(0,1]$}, \mbox{$\mX(0)=\hat\mX(0)\coloneqq\vect{0}$}, and \mbox{$\vy{(0)}= \vect{1}$}. Then under Assumption~\ref{assump:bounded-gradient}, the iterates of update rule~\eqref{eq:update-opt} satisfies the following property:
		\begin{align*}
			\Psi_{x}(t) \leq \frac{1632\,C^2 (C{+}1)^2\beta^2(\beta{+}1)^2n G^2\eta^2}{\omega^2\delta^4\kappa^2},
		\end{align*}
		where \mbox{\footnotesize $\Psi_{x}(t) \coloneqq \bbE\lVert\mZ(t{+}1){-}\overline{\mX}(t)\rVert_F^2$},
		when \mbox{$\gamma\coloneqq\frac{\omega\delta}{12\beta(\beta+1)(C+1)}$}.
	\end{lemma}
	
	\begin{proof}[Proof of Lemma~\ref{lem:diff-opt}]
		Based on the update rule in~\eqref{eq:update-opt}, we define two error functions \mbox{$\mcK(t)=\bbE\lVert\mX(t){-}\vphi\vect{1}^\top\mX(t)\rVert_F^2$} and \mbox{$\hat{\mcK}(t)=\bbE\lVert\hat\mX(t{+}1){-}\mX(t)\rVert_F^2$}.
		Note that the definition of $\mcK(t)$ differs from its counterpart in the proof of Theorem~\ref{thm:consensus}. By rewriting the recursion in~\eqref{eq:update-opt}, similar to~\eqref{eq:cons-proof-1}, we have
		\begin{align}\label{eq:opt-proof-1}
			\mX(t{+}1) &= \mB^{t{+}1} \mX(0)\nonumber\\
			& + \gamma\,\sum\limits_{s=0}^{t} \mB^s\left(\mW{-}\mI\right)\left[\hat\mX(t{-}s{+}1) - \mX(t{-}s)\right]\nonumber\\
			& - \eta\,\sum\limits_{s=0}^{t} \mB^s \,\partial\tilde{F}(\mZ{(t{-}s{+}1)},\vxi_{t{-}s{+}1}),
		\end{align}
		therefore, we have
		\begin{align}\label{eq:opt-proof-2}
			\mX(&t{+}1) - \vphi\vect{1}^\top \mX(t{+}1) \nonumber\\
			\overset{\eqref{eq:q-comp},\eqref{eq:opt-proof-1}}&{=} \left(\mB^{t{+}1}-\vphi\vect{1}^\top\right) \mX(0)\nonumber\\
			& + \gamma\,\sum\limits_{s=0}^{t} \left(\mB^s-\vphi\vect{1}^\top\right)\left(\mW{-}\mI\right)\left[\hat\mX(t{-}s{+}1) - \mX(t{-}s)\right]\nonumber\\
			& - \eta\,\sum\limits_{s=0}^{t} \left(\mB^s-\vphi\vect{1}^\top\right) \,\partial\tilde{F}(\mZ{(t{-}s{+}1)},\vxi_{t{-}s{+}1}).
		\end{align}
		Before proceeding with the proof, let us state some inequalities. For any set of $m$ matrices $\{\mA_i\}_{i=1}^{m}$ such that \mbox{$\mA_i\in\bbR^{n{\times}d}$}, matrix \mbox{$\mB\in\bbR^{n{\times}n}$}, and constant \mbox{$\alpha>0$}, the following properties hold: for all \mbox{$i,j\in[m]$:}
		\begingroup
		\allowdisplaybreaks
		\begin{subequations}\label{eq:frobenius}
			\begin{align}
				\lVert \mA_i + \mA_j \rVert_F^2 &\leq (1{+}\alpha)\lVert\mA_i\rVert_F^2 + (1{+}\alpha^{-1})\lVert\mA_j\rVert_F^2,\label{eq:frobenius-1}
				\\\lVert \mA_i + \mA_j \rVert_F &\leq \lVert\mA_i\rVert_F + \lVert\mA_j\rVert_F,\label{eq:frobenius-2}
				\\\lVert \mB\mA_i \rVert_F &\leq \lVert\mB\rVert\lVert\mA_i\rVert_F,\label{eq:frobenius-3}
				\\2\lVert \mA_i\rVert_F\lVert\mA_j\rVert_F &\leq \lVert\mA_i\rVert_F^2 + \lVert\mA_j\rVert_F^2,\label{eq:frobenius-4}
				\\(1-\alpha)&\left(1+\frac{\alpha}{2}\right)\leq 1-\frac{\alpha}{2}\label{eq:frobenius-5},
				\\(1-\alpha)&\left(1+\frac{2}{\alpha}\right)\leq \frac{2}{\alpha}\label{eq:frobenius-6},
				\\\left\lVert\sum\limits_{i=1}^m \mA_i\right\rVert^2 &\leq m \left(\sum\limits_{i=1}^m \lVert\mA_i\rVert^2\right).\label{eq:frobenius-7}
			\end{align}
		\end{subequations}
		\endgroup
		Moreover, for any set of $m$ vectors \mbox{$\{\va_i\}_{i{=}1}^m$}, where \mbox{$\va_i\in \bbR^d$}, the following property holds:
		\begin{align}\label{eq:norm-2-n}
			\left\lVert\sum\limits_{i=1}^m \va_i\right\rVert^2 &\leq m \sum\limits_{i=1}^m \left(\lVert\va_i\rVert^2\right).
		\end{align}
		Now, according to the stated inequalities and~\eqref{eq:opt-proof-1}, we have
		\begin{align}\label{eq:opt-proof-3}
			&\mcK(t{+}1)=\big\lVert\mX(t{+}1) - \vphi\vect{1}^\top \mX(t{+}1)\big\rVert_F^2\nonumber\\
			\overset{\eqref{eq:frobenius-7}}&{\leq}{ 2\gamma^2\Bigg\lVert\sum\limits_{s=0}^{t} \medmath{\left(\mB^s{-}\vphi\vect{1}^\top\right)\left(\mW{-}\mI\right)\left[\hat\mX(t{-}s{+}1) {-} \mX(t{-}s)\right]}\Bigg\rVert_F^2}\nonumber\\
			& + 2\eta^2\Bigg\lVert\sum\limits_{s=0}^{t} \medmath{\left(\mB^s {-} \vphi\vect{1}^\top\right) \,\partial\tilde{F}\left(\mZ{(t{-}s{+}1)},\vxi_{t{-}s{+}1}\right)}\Bigg\rVert_F^2\nonumber\\
			\overset{\eqref{eq:frobenius-2},\eqref{eq:frobenius-3}}&{\leq} 2\gamma^2\beta^2 C^2\left[\sum\limits_{s=0}^{t} \medmath{(1{-}\gamma\delta)^s \left\lVert\hat\mX(t{-}s{+}1) {-} \mX(t{-}s)\right\rVert_F}\right]^2\nonumber\\
			& + 2\eta^2 C^2\left[\sum\limits_{s=0}^{t} \medmath{(1{-}\gamma\delta)^s \left\lVert\partial\tilde{F}\left(\mZ{(t{-}s{+}1)},\vxi_{t{-}s{+}1}\right)\right\rVert_F}\right]^2,
		\end{align}
		where according to AM-GM inequality we also have the following two inequalities:
		\begin{align}\label{eq:opt-proof-4}
			&\left[\sum\limits_{s=0}^{t} \medmath{(1{-}\gamma\delta)^s \left\lVert\hat\mX(t{-}s{+}1) {-} \mX(t{-}s)\right\rVert_F}\right]^2 \nonumber\\
			& = \sum\limits_{s=0}^{t}\sum\limits_{\hat{s}=0}^{t} \medmath{(1{-}\gamma\delta)^{s+\hat{s}} \left\lVert\hat\mX(t{-}s{+}1) {-} \mX(t{-}s)\right\rVert_F}\nonumber\\
			&\qquad\qquad\qquad\qquad\qquad\medmath{\times\left\lVert\hat\mX(t{-}\hat{s}{+}1) {-} \mX(t{-}\hat{s})\right\rVert_F}\nonumber\\
			\overset{\eqref{eq:frobenius-4}}&{\leq}\frac{1}{2}\sum\limits_{s=0}^{t}\sum\limits_{\hat{s}=0}^{t} \medmath{(1{-}\gamma\delta)^{s+\hat{s}} \Bigg[\left\lVert\hat\mX(t{-}s{+}1) {-} \mX(t{-}s)\right\rVert_F^2}\nonumber\\
			&\qquad\qquad\qquad\qquad\qquad\medmath{+\left\lVert\hat\mX(t{-}\hat{s}{+}1) {-} \mX(t{-}\hat{s})\right\rVert_F^2\Bigg]}\nonumber\\
			\overset{}&{=}\sum\limits_{s=0}^{t}\medmath{(1{-}\gamma\delta)^{s} \left\lVert\hat\mX(t{-}s{+}1) {-} \mX(t{-}s)\right\rVert_F^2 \sum\limits_{\hat{s}=0}^{t} (1{-}\gamma\delta)^{\hat{s}}}\nonumber\\
			\overset{}&{\leq}\frac{1}{\gamma\delta}\sum\limits_{s=0}^{t}\medmath{(1{-}\gamma\delta)^{s} \left\lVert\hat\mX(t{-}s{+}1) {-} \mX(t{-}s)\right\rVert_F^2},
		\end{align}
		as well as
		\begin{align}\label{eq:opt-proof-5}
			&\left[\sum\limits_{s=0}^{t} \medmath{(1{-}\gamma\delta)^s \left\lVert\partial\tilde{F}\left(\mZ{(t{-}s{+}1)},\vxi_{t{-}s{+}1}\right)\right\rVert_F}\right]^2\nonumber\\
			& = \sum\limits_{s=0}^{t}\sum\limits_{\hat{s}=0}^{t} \medmath{(1{-}\gamma\delta)^{s+\hat{s}} \left\lVert\partial\tilde{F}\left(\mZ{(t{-}s{+}1)},\vxi_{t{-}s{+}1}\right)\right\rVert_F}\nonumber\\
			&\qquad\qquad\qquad\qquad\qquad\medmath{\times\left\lVert\partial\tilde{F}\left(\mZ{(t{-}\hat{s}{+}1)},\vxi_{t{-}\hat{s}{+}1}\right)\right\rVert_F}\nonumber\\
			\overset{\text{\footnotesize\eqref{eq:frobenius-4}}}&{\leq}\frac{1}{2}\sum\limits_{s=0}^{t}\sum\limits_{\hat{s}=0}^{t} \medmath{(1{-}\gamma\delta)^{s+\hat{s}} \Bigg[\left\lVert\partial\tilde{F}\left(\mZ{(t{-}s{+}1)},\vxi_{t{-}s{+}1}\right)\right\rVert_F^2}\nonumber\\
			&\qquad\qquad\qquad\qquad\qquad\medmath{+\left\lVert\partial\tilde{F}\left(\mZ{(t{-}\hat{s}{+}1)},\vxi_{t{-}\hat{s}{+}1}\right)\right\rVert_F^2\Bigg]}\nonumber\\
			\overset{}&{=}\sum\limits_{s=0}^{t}\medmath{(1{-}\gamma\delta)^{s} \left\lVert\partial\tilde{F}\left(\mZ{(t{-}s{+}1)},\vxi_{t{-}s{+}1}\right)\right\rVert_F^2 \sum\limits_{\hat{s}=0}^{t} (1{-}\gamma\delta)^{\hat{s}}}\nonumber\\
			\overset{}&{\leq}\frac{1}{\gamma\delta}\sum\limits_{s=0}^{t}\medmath{(1{-}\gamma\delta)^{s} \left\lVert\partial\tilde{F}\left(\mZ{(t{-}s{+}1)},\vxi_{t{-}s{+}1}\right)\right\rVert_F^2}\nonumber\\
			\overset{\text{\footnotesize Assump.~\ref{assump:bounded-gradient}}}&{\leq} \frac{nG^2}{\gamma^2\delta^2}.
		\end{align}
		Furthermore, according to~\eqref{eq:q-comp}, and similar to~\eqref{eq:cons-proof-4}, we have
		\begin{align}\label{eq:opt-proof-6}
			&\bbE\Big\lVert\mX(t{+}1) {-} \hat{\mX}(t{+}2)\Big\rVert_F^2\nonumber\\
			&\leq (1{-}\omega)\Big\lVert \left[\left(1{+}\gamma\right)\mI {-} \gamma\mW\right]\left[\mX(t){-}\hat{\mX}(t{+}1)\right]\nonumber\\
			&\qquad\qquad\quad+\gamma\,(\mW-\mI)\left[\mX(t){-}\vphi\vect{1}^\top\mX(t)\right]\nonumber\\
			&\qquad\qquad\quad- \eta\,\partial\tilde{F}(\mZ{(t{+}1)},\vxi_{t{+}1})\Big\rVert_F^2\nonumber\\
			\overset{\eqref{eq:frobenius-1}}&{\leq} (1{-}\omega)\left(1{+}\frac{\omega}{2}\right)\Big\lVert\left[\left(1{+}\gamma\right)\mI {-} \gamma\mW\right]\left[\mX(t){-}\hat{\mX}(t{+}1)\right]\Big\rVert_F^2\nonumber\\
			&+(1{-}\omega)\left(1{+}\frac{2}{\omega}\right)\Big\lVert\gamma\,(\mW-\mI)\left[\mX(t){-}\vphi\vect{1}^\top\mX(t)\right]\nonumber\\
			&\qquad\qquad\qquad\quad- \eta\,\partial\tilde{F}(\mZ{(t{+}1)},\vxi_{t{+}1})\Big\rVert_F^2\nonumber\\
			\overset{\eqref{eq:frobenius-5},\eqref{eq:frobenius-6}}&{\leq} \left(1{-}\frac{\omega}{2}\right)\Big\lVert\left[\left(1{+}\gamma\right)\mI {-} \gamma\mW\right]\left[\mX(t){-}\hat{\mX}(t{+}1)\right]\Big\rVert_F^2\nonumber\\
			&+\frac{4}{\omega}\Big\lVert\gamma\,(\mW-\mI)\left[\mX(t){-}\vphi\vect{1}^\top\mX(t)\right]\Big\rVert_F^2\nonumber\\
			&+\frac{4}{\omega}\Big\lVert\eta\,\partial\tilde{F}(\mZ{(t{+}1)},\vxi_{t{+}1})\Big\rVert_F^2\nonumber\\
			\overset{\eqref{eq:frobenius-3}}&{\leq} \left(1{-}\frac{\omega}{2}\right)(1{+}\gamma\beta)^2\Big\lVert\mX(t){-}\hat{\mX}(t{+}1)\Big\rVert_F^2\nonumber\\
			&+\frac{4}{\omega}\gamma^2\beta^2\Big\lVert\mX(t){-}\vphi\vect{1}^\top\mX(t)\Big\rVert_F^2\nonumber\\
			&+\frac{4n\eta^2G^2}{\omega}.
		\end{align}
		Therefore, due to~\eqref{eq:opt-proof-3},~\eqref{eq:opt-proof-4},~\eqref{eq:opt-proof-5}, and~\eqref{eq:opt-proof-6}, the following two inequalities hold given parameter $\gamma$ introduced in Lemma~\ref{lem:diff-opt}:
		\begingroup
		\allowdisplaybreaks
		\begin{subequations}\label{eq:opt-inequalities}
			\begin{align}
				\mcK(t{+}1) &\leq \frac{2\gamma\beta^2 C^2}{\delta}\sum\limits_{s{=}0}^{t} (1{-}\gamma\delta)^s \hat{\mcK}(t{-}s) + \frac{2n\eta^2C^2G^2}{\gamma^2\delta^2},\label{eq:opt-inequalities-a}\\
				\vspace{0.5em}
				\hat{\mcK}(t{+}1) &\leq \frac{4\gamma^2\beta^2}{\omega}\mcK(t) + \left(1{-}\frac{\omega}{2}\right)(1{+}\gamma\beta)^2 \hat{\mcK}(t) +\frac{4n\eta^2G^2}{\omega}.\label{eq:opt-inequalities-b}
			\end{align}
		\end{subequations}
		\endgroup
		Then, according to the inequalities in~\eqref{eq:opt-inequalities}, and by applying induction, similar to the proof of Theorem~\ref{thm:consensus}, it is sufficient to show that for any compression ratio \mbox{$\omega\in(0,1]$}, under the choice of $\gamma$ in Lemma~\ref{lem:diff-opt}, the following inequality holds:
		\begin{align}\label{eq:opt-K-hat-inequality}
			\hat{\mcK}(t) \leq \frac{128nG^2\max\{2\beta^2C^2,\delta^2\}}{\omega^2\delta^2} \eta^2 \coloneqq b_1\eta^2.
		\end{align}
		First, note that the base of induction holds. Moreover, assume that~\eqref{eq:opt-K-hat-inequality} holds for $t=0,1,\dots,T$. Due to~\eqref{eq:opt-inequalities}, we have the following
		\begin{align}\label{eq:opt-constant-1}
			\hat{\mcK}(t{+}1) &\leq \left(1{-}\frac{\omega}{2}\right)(1{+}\gamma\beta)^2 \hat{\mcK}(t) +\frac{4n\eta^2G^2}{\omega}\nonumber\\
			& +\frac{4\gamma^2\beta^2}{\omega}\Bigg[\frac{2\gamma\beta^2 C^2}{\delta}\sum\limits_{s{=}0}^{t{-}1} (1{-}\gamma\delta)^s \hat{\mcK}(t{-}s{-}1)\nonumber\\
			& \qquad\qquad+\frac{2n\eta^2C^2G^2}{\gamma^2\delta^2}\Bigg]\nonumber\\
			&\leq \left(1{-}\frac{\omega}{2}
			\right)(1{+}\gamma\beta)^2 \hat{\mcK}(t)\nonumber\\
			& +\frac{8\gamma^3\beta^4C^2}{\omega\delta}\sum\limits_{s{=}0}^{t{-}1} (1{-}\gamma\delta)^s \hat{\mcK}(t{-}s{-}1)\nonumber\\
			& +\frac{8n\beta^2C^2G^2 + 4n\delta^2G^2}{\omega\delta^2}\eta^2\nonumber\\
			\overset{\eqref{eq:opt-K-hat-inequality}}&{\leq} \left(1{-}\frac{\omega}{2}
			\right)(1{+}\gamma\beta)^2 b_1\eta^2\nonumber\\
			& +\frac{8\gamma^3\beta^4C^2}{\omega\delta}b_1\eta^2\sum\limits_{s{=}0}^{t{-}1} (1{-}\gamma\delta)^s\nonumber\\
			& +\frac{8nG^2\max\left\{2\beta^2C^2, \delta^2\right\}}{\omega\delta^2}\eta^2\nonumber\\
			&\leq \left(1{-}\frac{\omega}{2}
			\right)(1{+}\gamma\beta)^2 b_1\eta^2 + \frac{8\beta^4C^2}{\omega\delta^2}b_1\gamma^2\eta^2\nonumber\\
			& +\frac{8nG^2\max\left\{2\beta^2C^2, \delta^2\right\}}{\omega\delta^2}\eta^2,
		\end{align}
		therefore, it is sufficient to show that:
		\begin{align}\label{eq:opt-constant-2}
			\left(1{-}\frac{\omega}{2}
			\right)(1{+}\gamma\beta)^2 b_1 & + \frac{8\beta^4C^2}{\omega\delta^2}\gamma^2 b_1\nonumber\\
			& + \frac{8nG^2\max\left\{2\beta^2C^2, \delta^2\right\}}{\omega\delta^2} \leq b_1.
		\end{align}
		First of all, note that
		\begin{align}\label{eq:opt-constant-3}
			\frac{8nG^2\max\left\{2\beta^2C^2, \delta^2\right\}}{\omega\delta^2} = \frac{\omega b_1}{16}.
		\end{align}
		Moreover, under the choice of $\gamma$ in Lemma~\ref{lem:diff-opt}
		\begin{align}\label{eq:opt-constant-4}
			\gamma = \frac{\omega\delta}{12\beta(\beta+1)(C+1)} \leq \min\left\{\frac{\omega}{8\beta},\frac{\omega\delta}{12\beta^2C}\right\},
		\end{align}
		thus
		\begin{align}\label{eq:opt-constant-5}
			\gamma^2 \leq \frac{\omega^2\delta^2}{144\beta^4C^2} < \frac{\omega^2\delta^2}{128\beta^4C^2} \Rightarrow
			\frac{8\beta^4C^2}{\omega\delta^2}\gamma^2 < \frac{\omega}{16},
		\end{align}
		and
		\begin{align}\label{eq:opt-constant-6}
			\gamma \leq \frac{\omega}{8\beta} \Rightarrow 1+\gamma\beta \leq 1+\frac{\omega}{8} < 1+\frac{\omega}{4} \overset{\eqref{eq:frobenius-5}}&{\Rightarrow} \nonumber\\
			\left(1-\frac{\omega}{2}\right)(1+\gamma\beta)^2 \overset{\eqref{eq:frobenius-5}}{<} 1-\frac{\omega}{8} .&
		\end{align}
		Then, according to~\eqref{eq:opt-constant-4},~\eqref{eq:opt-constant-5}, and~\eqref{eq:opt-constant-6}, the inequality in~\eqref{eq:opt-constant-2} holds.
		\begin{align}\label{eq:opt-mat-ineq}
			\bbE\lVert\mZ(t{+}1){-}\overline{\mX}(t)\rVert_F^2 & = \sum_{i=1}^n \bbE\left\lVert \vz_i(t{+}1) - \overline{\vx}(t) \right\rVert^2\nonumber\\
			& = \sum_{i=1}^n \bbE\left\lVert \frac{\vu_i(t{+}1) {-}  y_i(t{+}1)\overline{\vx}(t)}{y_i(t{+}1)} \right\rVert^2\nonumber\\
			\overset{\eqref{eq:mixing-matrix}}&{\leq}\sum\limits_{i=1}^n \frac{\bbE\left\lVert \vu_i(t{+}1) {-} [\mW^{t{+}1}\vect{1}]_i\overline{\vx}(t) \right\rVert^2}{\kappa^2}\nonumber\\
			& = \frac{1}{\kappa^2} \bbE\Big\lVert\mU(t{+}1) {-} \mW^{t{+}1} \frac{\vect{1}\vect{1}^\top}{n}\mX(t)\Big\rVert_F^2\nonumber\\
			= & \frac{1}{\kappa^2} \bbE\Big\lVert\mU(t{+}1) {-} \mW^{t{+}1} \frac{\vect{1}\vect{1}^\top}{n}\mU(t{+}1)\Big\rVert_F^2,
		\end{align}
		where we use the fact that \mbox{$\frac{\vect{1}\vect{1}^\top}{n}\mU(t{+}1)=\frac{\vect{1}\vect{1}^\top}{n}\mX(t)$}.
		%
		%
		Again, using the update rule in~\eqref{eq:update-opt}, inequality~\eqref{eq:opt-K-hat-inequality}, and Cauchy-Schwarz, we have the following:
		\begin{align}\label{eq:bounded-U-average}
			&\bbE\Big\lVert\mU(t{+}1) {-} \mW^{t{+}1} \frac{\vect{1}\vect{1}^\top}{n}\mU(t{+}1)\Big\rVert_F^2\nonumber\\
			\overset{\eqref{eq:frobenius-7}}&{\leq} 3\gamma^2 \left\lVert\sum_{s=0}^t (\mB^s-\vphi \vect{1}^T)(\mW{-}\mI)\left[\hat{\mX}(t{-}s{+}1) {-} \mX(t{-}s)\right]\right\rVert_F^2\nonumber\\
			&+ 3\eta^2 \left\lVert\sum_{s=0}^t (\mB^s{-}\vphi \vect{1}^T)\,\partial\tilde{F}\big(\mZ{(t{-}s{+}1\big)},\vxi_{t{-}s{+}1})\right\rVert_F^2\nonumber\\
			&+ 3\eta^2 \left\lVert\sum_{s=0}^t (\mW^{t{+}1}{-}\vphi \vect{1}^T)\,\partial\tilde{F}\big(\mZ{(t{-}s{+}1)},\vxi_{t{-}s{+}1}\big)\right\rVert_F^2\nonumber\\
			\overset{\eqref{eq:lazy-matrix}}&{\leq} 3C^2\gamma^2\beta^2 \left(\sum_{s=0}^t(1{-}\gamma\delta)^s\left\lVert\hat{\mX}(t{-}s{+}1) {-} \mX(t{-}s)\right\rVert_F\right)^2\nonumber\\
			& + 3C^2\eta^2 \left(\sum_{s=0}^t(1{-}\gamma\delta)^s\left\lVert\partial\tilde{F}\big(\mZ{(t{-}s{+}1\big)},\vxi_{t{-}s{+}1}))\right\rVert_F\right)^2\nonumber\\
			& + \frac{3C^2nG^2}{\delta}\eta^2\nonumber\\
			\overset{\text{\tiny similar to }\eqref{eq:opt-proof-5}}&{\leq} \frac{3\,C^2 \beta^2 b_1}{\delta^2}\eta^2 + \frac{3\,C^2nG^2}{\gamma^2\delta^2}\eta^2 + \frac{3C^2nG^2}{\delta}\eta^2\nonumber\\
			\overset{\eqref{eq:opt-constant-4}}&{\leq} \frac{1632\,C^2 (C{+}1)^2\beta^2(\beta{+}1)^2n G^2}{\omega^2\delta^4}\eta^2,
		\end{align}
		which concludes the proof.
	\end{proof}
	
	Lemma~\ref{lem:diff-opt} indicates that the agents can control their agreement (consensus) \mbox{$\Psi_{x}(t)$} with a proper choice of optimization stepsize $\eta$ while trying to find a (sub)optimal solution for the optimization problem. In other words, the upper bound on \mbox{$\Psi_{x}(t)$} shows the level of coordination between the agents in the corresponding decentralized setup.
	
	We can use the result in Lemma~\ref{lem:diff-opt} to show the convergence of the following theorems. We first state our convergence result for decentralized smooth and strongly convex stochastic optimization over a directed network with arbitrary compressed communication.

	\begin{theorem}[Smooth and Strongly Convex Stochastic Optimization]\label{thm:strongly-convex}
		Let the compression operator $Q$ satisfy~\eqref{eq:q-comp}, \mbox{$\mX(0)=\hat\mX(0)\coloneqq\vect{0}$}, and \mbox{$\vy{(0)}= \vect{1}$}. Then, under Assumptions~\ref{assump:bounded-variance}-\ref{assump:strong-convex}, the iterates of
		update rule~\eqref{eq:update-opt}
		have the following property: for any \mbox{$\omega\in(0,1]$}, and \mbox{$T\geq 64L^2/\mu^2$},
		\begin{align*}
			\bbE f\Bigg(\frac{1}{nS(T)}\sum_{t=0}^{T-1}p^t\sum_{i=1}^{n}\vx_i(T{-}t{-}1)\Bigg) &- f(\vx^\star) \\
			\qquad \leq \frac{C_1\log T}{nT} + \frac{C_2}{T} + \frac{C_3 (\log T)^2}{T^2},
		\end{align*}
		where \mbox{$C_3\coloneqq \frac{13056\,C^2(C{+}1)^2\beta^2(\beta{+}1)^2G^2 L(L{+}1)}{\mu^2\omega^2\delta^4\kappa^2}$}, \mbox{$C_1\coloneqq \frac{2\sigma^2}{\mu}$}, and \mbox{$C_2\coloneqq \frac{\mu\lVert\vx^\star\rVert^2}{4}$}, when \mbox{$p\coloneqq 1{-}\frac{\log T}{T}$}, \mbox{\small $S(T)=\sum\limits_{t=0}^{T-1} p^t$}, \mbox{$\eta\coloneqq \frac{2\log T}{\mu T}$}, and \mbox{$\gamma$} as in Lemma~\ref{lem:diff-opt}.
	\end{theorem}
	
	\begin{proof}[Proof of Theorem~\ref{thm:strongly-convex}]
		Let \mbox{$\vx^\star\in\bbR^d$} be the minimizer of Problem~\eqref{eq:opt}. Then,
		\begin{align}\label{eq:opt-scvx-1}
			\overline{\vx}(t{+}1){-}\vx^\star = \overline{\vx}(t) -\vx^\star &- \frac{\eta}{n} \sum\limits_{i{=}1}^{n} \nabla\tilde{f}_i(\vz_{i}(t{+}1),\vxi_{i,t{+}1}) \Rightarrow\nonumber\\
			\overline{\vx}(t{+}1){-}\vx^\star = \overline{\vx}(t) -\vx^\star &- \frac{\eta}{n} \sum\limits_{i{=}1}^{n} \nabla f_i(\vz_{i}(t{+}1)) \nonumber\\
			+ \frac{\eta}{n} \sum\limits_{i{=}1}^{n} \nabla f_i(\vz_{i}(t{+}1)) &- \frac{\eta}{n} \sum\limits_{i{=}1}^{n} \nabla\tilde{f}_i(\vz_{i}(t{+}1),\vxi_{i,t{+}1}),
		\end{align}
		where by applying the norm operator and taking expectation on the second moment, we have
		\begin{align}\label{eq:opt-scvx-2}
			\bbE&\left\lVert\overline{\vx}(t{+}1){-}\vx^\star\right\rVert^2\nonumber\\
			&= \bbE\left\lVert\overline{\vx}(t) -\vx^\star - \frac{\eta}{n} \sum\limits_{i{=}1}^{n} \nabla f_i(\vz_{i}(t{+}1))\right\rVert^2 \nonumber\\
			&+ \frac{\eta^2}{n^2} \bbE\left\lVert\sum\limits_{i{=}1}^{n} \nabla f_i(\vz_{i}(t{+}1)) - \sum\limits_{i{=}1}^{n} \nabla\tilde{f}_i(\vz_{i}(t{+}1),\vxi_{i,t{+}1})\right\rVert^2\nonumber\\
			&+\frac{2\eta}{n}\Bigg\langle \overline{\vx}(t) -\vx^\star - \frac{\eta}{n} \sum\limits_{i{=}1}^{n} \nabla f_i(\vz_{i}(t{+}1)),\nonumber\\
			&\qquad\quad\sum\limits_{i{=}1}^{n} \nabla f_i(\vz_{i}(t{+}1)) - \sum\limits_{i{=}1}^{n} \nabla\tilde{f}_i(\vz_{i}(t{+}1),\vxi_{i,t{+}1})\Bigg\rangle\nonumber\\
			\overset{\text{\footnotesize Assump.~\ref{assump:bounded-variance}},\eqref{eq:norm-2-n}}&{\leq} \bbE\left\lVert\overline{\vx}(t){-}\vx^\star {-} \frac{\eta}{n} \sum\limits_{i{=}1}^{n} \nabla f_i(\vz_{i}(t{+}1))\right\rVert^2 + \frac{\eta^2\sigma^2}{n}.
		\end{align}
		Moreover, we have:
		\begin{align}\label{eq:opt-scvx-3}
			\bbE\Bigg\lVert\overline{\vx}(t) &-\vx^\star - \frac{\eta}{n} \sum\limits_{i{=}1}^{n} \nabla f_i(\vz_{i}(t{+}1))\Bigg\rVert^2\nonumber\\
			&= \bbE\left\lVert\overline{\vx}(t) -\vx^\star\right\rVert^2 + \eta^2\, \bbE\left\lVert\frac{1}{n} \sum\limits_{i{=}1}^{n} \nabla f_i(\vz_{i}(t{+}1))\right\rVert^2 \nonumber\\
			&-2\eta \bbE\left\langle\overline{\vx}(t) -\vx^\star,\frac{1}{n} \sum\limits_{i{=}1}^{n} \nabla f_i(\vz_{i}(t{+}1))\right\rangle,
		\end{align}
		where the second term of the upper bound in~\eqref{eq:opt-scvx-3} can be bounded by using Assumptions~\ref{assump:l-smooth} and~\ref{assump:convex}, as follows:
		\begin{align}\label{eq:opt-scvx-4}
			f(\vx^\star) &\leq f(\overline{\vx}(t)) + \left\langle \nabla f(\vx^\star), \overline{\vx}(t)-\vx^\star\right\rangle \nonumber\\
			&- \frac{1}{2L} \left\lVert\nabla f(\overline{\vx}(t)) - \nabla f(\vx^\star)\right\rVert^2 \Rightarrow \nonumber\\
			\left\lVert\nabla f(\overline{\vx}(t))\right\rVert^2 &\leq 2L\left(f(\overline{\vx}(t))-f(\vx^\star)\right),
		\end{align}
		thus, we have
		\begin{align}\label{eq:opt-scvx-5}
			\bbE\Bigg\lVert\frac{1}{n}\sum\limits_{i{=}1}^n &\nabla f_i(\vz_{i}(t{+}1))\Bigg\rVert^2\nonumber\\
			\overset{\eqref{eq:norm-2-n}}&{\leq} 2\bbE\left\lVert\frac{1}{n}\sum\limits_{i{=}1}^n \left[\nabla f_i(\vz_{i}(t{+}1)) - \nabla f_i(\overline{\vx}(t))\right]\right\rVert^2\nonumber\\
			& + 2\bbE\left\lVert\frac{1}{n}\sum\limits_{i{=}1}^n \nabla f_i(\overline{\vx}(t))\right\rVert^2\nonumber\\
			& = \frac{2}{n^2}\bbE\left\lVert\sum\limits_{i{=}1}^n \left[\nabla f_i(\vz_{i}(t{+}1)) - \nabla f_i(\overline{\vx}(t))\right]\right\rVert^2\nonumber\\
			& + 2\bbE\left\lVert\nabla f(\overline{\vx}(t))\right\rVert^2\nonumber\\
			\overset{\eqref{eq:norm-2-n},\eqref{eq:opt-scvx-4}}&{\leq} \frac{2}{n}\sum\limits_{i{=}1}^n\bbE\left\lVert \nabla f_i(\vz_{i}(t{+}1)) - \nabla f_i(\overline{\vx}(t))\right\rVert^2\nonumber\\
			& + 4L\left(\bbE f(\overline{\vx}(t))-f(\vx^\star)\right)\nonumber\\
			\overset{\text{\footnotesize Assump.~\ref{assump:l-smooth}}}&{\leq} \frac{2L^2}{n}\sum\limits_{i{=}1}^n\bbE\left\lVert \vz_{i}(t{+}1) - \overline{\vx}(t)\right\rVert^2\nonumber\\
			& + 4L\left(\bbE f(\overline{\vx}(t))-f(\vx^\star)\right).
		\end{align}
		Now, consider the following inequality, for all $i\in[n]$:
		\begin{align}\label{eq:opt-scvx-6}
			\langle\overline{\vx}(t) &{-}\vx^\star, \nabla f_i(\vz_{i}(t{+}1))\rangle\nonumber\\
			&= \left\langle\overline{\vx}(t) {-}\vz_{i}(t{+}1), \nabla f_i(\vz_{i}(t{+}1))\right\rangle\nonumber\\
			&+ \left\langle\vz_{i}(t{+}1) {-}\vx^\star,\nabla f_i(\vz_{i}(t{+}1))\right\rangle\nonumber\\
			\overset{\text{\footnotesize Assump.~\ref{assump:l-smooth}-\ref{assump:strong-convex}}}&{\geq} f_i(\overline{\vx}(t))- f_i(\vz_{i}(t{+}1))- \frac{L}{2} \left\lVert\overline{\vx}(t)-\vz_{i}(t{+}1)\right\rVert^2\nonumber\\
			& + f_i(\vz_{i}(t{+}1)) - f_i(\vx^\star) + \frac{\mu}{2} \left\lVert\vz_{i}(t{+}1)-\vx^\star\right\rVert^2\nonumber\\
			\overset{\eqref{eq:norm-2-n}}&{\geq} f_i(\overline{\vx}(t))- f_i(\vx^\star) - \frac{L}{2} \left\lVert\overline{\vx}(t)-\vz_{i}(t{+}1)\right\rVert^2\nonumber\\
			& + \frac{\mu}{2} \left[\frac{1}{2}\left\lVert\overline{\vx}(t)-\vx^\star\right\rVert^2 - \left\lVert\vz_{i}(t{+}1)-\overline{\vx}(t)\right\rVert^2\right]\nonumber\\
			&= f_i(\overline{\vx}(t))- f_i(\vx^\star) - \frac{L+\mu}{2} \left\lVert\overline{\vx}(t)-\vz_{i}(t{+}1)\right\rVert^2\nonumber\\
			& + \frac{\mu}{4}\left\lVert\overline{\vx}(t)-\vx^\star\right\rVert^2
		\end{align}
		thus, we have the following bound for the third term in~\eqref{eq:opt-scvx-3}:
		\begin{align}\label{eq:opt-scvx-7}
			\Bigg\langle\overline{\vx}(t) &{-}\vx^\star, \frac{1}{n}\sum\limits_{i{=}1}^n\nabla f_i(\vz_{i}(t{+}1))\Bigg\rangle\nonumber\\
			&= \frac{1}{n}\sum\limits_{i{=}1}^n\Bigg\langle\overline{\vx}(t) {-}\vx^\star, \nabla f_i(\vz_{i}(t{+}1))\Bigg\rangle\nonumber\nonumber\\
			\overset{\eqref{eq:opt-scvx-6}}&{\geq} f(\overline{\vx}(t))- f(\vx^\star) - \frac{L+\mu}{2n}\sum\limits_{i{=}1}^{n} \left\lVert\overline{\vx}(t)-\vz_{i}(t{+}1)\right\rVert^2\nonumber\\
			& + \frac{\mu}{4}\left\lVert\overline{\vx}(t)-\vx^\star\right\rVert^2.
		\end{align}
		Finally, according to~\eqref{eq:opt-scvx-2},~\eqref{eq:opt-scvx-3},~\eqref{eq:opt-scvx-5}, and~\eqref{eq:opt-scvx-7}, the following inequality holds:
		\begin{align}\label{eq:strongly-convex-base-ineq}
			\bbE\lVert\overline{\vx}&(t{+}1){-}\vx^\star\rVert^2\leq \left(1{-}\frac{\eta\mu}{2}\right)\bbE\lVert\overline{\vx}(t){-}\vx^\star\rVert^2 + \frac{\eta^2\sigma^2}{n}\nonumber\\
			&{-}2\eta(1{-}2L\eta)(\bbE f(\overline{\vx}(t)){-}f(\vx^\star)) + \eta\frac{2\eta L^2{+}L{+}\mu}{n}\Psi_{x}(t).
		\end{align}
		Now, let us consider inequality~\eqref{eq:strongly-convex-base-ineq} for \mbox{$t=0,1,\dots,T{-}1$}, and fix \mbox{$p=1{-}\frac{\eta\mu}{2}$}. Then, by taking a weighted average of these inequalities with weight $p^t$, we have
		\begin{align}\label{eq:strongly-convex-final-ineq}
			2(1&{-}2L\eta)\left[\bbE f\left(\frac{1}{\mathsmaller{\sum\limits_{s=0}^{T{-}1}p^s}}\sum\limits_{t=0}^{T{-}1}p^t\overline{\vx}(T{-}t{-}1)\right)-f(\vx^\star)\right]\nonumber\\
			&\leq 2(1{-}2L\eta)\left[\frac{1}{\mathsmaller{\sum\limits_{s=0}^{T{-}1}p^s}}\sum\limits_{t=0}^{T{-}1}p^t \bbE f(\overline{\vx}(T{-}t{-}1))-f(\vx^\star)\right]\nonumber\\
			&\leq \frac{\mu p^T \lVert\vx^\star\rVert^2}{2(1{-}p^T)} + \frac{\eta \sigma^2}{n}\nonumber\\
			& + \frac{(2\eta L^2{+}L{+}\mu)(1632\,C^2 (C{+}1)^2\beta^2(\beta{+}1)^2n G^2)}{n\omega^2\delta^4\kappa^2}\eta^2,
		\end{align}
		which for $\eta\coloneqq\frac{2\log T}{\mu T}$, we can conclude the statement in Theorem~\ref{thm:strongly-convex}. Note that the choice of optimization stepsize $\eta$ makes a trade-off between the order of the first and second expressions on the right-hand side of~\eqref{eq:strongly-convex-final-ineq}, based on $T$. With a very small $\eta$, the second and third expressions on the right-hand side converge faster to zero, while $p^T$ in the first expression will require more rounds to converge to zero.
	\end{proof}
	
	Theorem~\ref{thm:strongly-convex} suggests a sublinear rate \mbox{$\mcO((\log T)/T)$}, that only differs in a logarithmic term compared to CHOCO-SGD for undirected graphs~\cite[Theorem~4]{koloskova2019decentralized}. We consider a constant optimization stepsize $\eta$, while the authors of~\cite{koloskova2019decentralized} select a decreasing sequence.

	Before stating the next theorem, note that the convergence rates presented here are based on a (possibly weighted) average of the variables across the time $t$ and agents $i$. This enables a more straightforward presentation for our analysis. Nevertheless, using Lemma~\ref{lem:diff-opt}, similar results can be shown for the variable $\vz_i(t)$.

	\begin{theorem}[Smooth and Convex Stochastic Optimization]\label{thm:convex}
		Let the compression operator $Q$ satisfy~\eqref{eq:q-comp}, $\mX(0)=\hat\mX(0)=\vect{0}$, and \mbox{$\vy{(0)}= \vect{1}$}. Then, under Assumptions\ref{assump:bounded-variance}-\ref{assump:convex}, the iterates of 
		update rule~\eqref{eq:update-opt}
		have the following property: for any \mbox{$\omega\in(0,1]$}, and \mbox{$T\geq n$},
		\begin{align*}
			&\bbE f\left(\frac{1}{nT}\sum_{t=0}^{T-1}\sum_{i=1}^{n}\vx_i(t)\right) - f(\vx^\star)\leq \frac{C_4}{\sqrt{nT}} + \frac{C_5}{T},
		\end{align*}
		where \mbox{$C_4\coloneqq \frac{16L^2\lVert\vx^\star\rVert^2+\sigma^2}{4L}$}, \mbox{$C_5\coloneqq \frac{153\,C^2(C{+}1)^2\beta^2(\beta{+}1)^2G^2n}{L\omega^2\delta^4\kappa^2}$}, \vspace{0.1em} when \mbox{$\eta\coloneqq \frac{\sqrt{n}}{4L\sqrt{T}}$}, and \mbox{$\gamma$} as in Lemma~\ref{lem:diff-opt}.
	\end{theorem}
	
	\begin{proof}[Proof of Theorem~\ref{thm:convex}]
		Let \mbox{$\vx^\star\in\bbR^d$} be a global minimum of $f(.)$, i.e., \mbox{$f^\star = f(\vx^\star)$}. According to the update rule in~\eqref{eq:update-opt}
		\begin{align}\label{eq:opt-cvx-1}
			\overline{\vx}(t{+}1){-}\vx^\star = \overline{\vx}(t) &-\vx^\star - \frac{\eta}{n} \sum\limits_{i{=}1}^{n} \nabla\tilde{f}_i(\vz_{i}(t{+}1),\vxi_{i,t{+}1}) \Rightarrow\nonumber\\
			\bbE \lVert\overline{\vx}(t{+}1){-}\vx^\star\rVert^2 &= \bbE \lVert\overline{\vx}(t){-}\vx^\star\rVert^2\nonumber\\
			- \frac{2\eta}{n} &\sum\limits_{i{=}1}^n \bbE \left\langle\nabla\tilde{f}_i(\vz_{i}(t{+}1),\vxi_{i,t{+}1}),\overline{\vx}(t){-}\vx^\star\right\rangle\nonumber\\
			&+ \eta^2 \,\bbE\, \left\lVert \frac{1}{n}\sum\limits_{i{=}1}^n \nabla\tilde{f}_i(\vz_{i}(t{+}1),\vxi_{i,t{+}1})\right\rVert^2\nonumber\\
			&= \bbE \lVert\overline{\vx}(t){-}\vx^\star\rVert^2\nonumber\\
			&- \frac{2\eta}{n} \sum\limits_{i{=}1}^n \bbE \left\langle\nabla f_i(\vz_{i}(t{+}1)),\overline{\vx}(t){-}\vx^\star\right\rangle\nonumber\\
			&+ \eta^2 \,\bbE\, \left\lVert \frac{1}{n}\sum\limits_{i{=}1}^n \nabla\tilde{f}_i(\vz_{i}(t{+}1),\vxi_{i,t{+}1})\right\rVert^2.
		\end{align}
		First, we have
		\begin{align}\label{eq:opt-cvx-2}
			\bbE&\, \left\lVert\frac{1}{n}\sum\limits_{i{=}1}^n \nabla\tilde{f}_i(\vz_{i}(t{+}1),\vxi_{i,t{+}1})\right\rVert^2\nonumber\\
			\overset{}&{=}  \bbE\, \left\lVert\frac{1}{n}\sum\limits_{i{=}1}^n \left( \nabla\tilde{f}_i(\vz_{i}(t{+}1),\vxi_{i,t{+}1})-\nabla f_i(\vz_{i}(t{+}1))\right)\right\rVert^2\nonumber\\
			& + \bbE\, \left\lVert\frac{1}{n}\sum\limits_{i{=}1}^n \nabla f_i(\vz_{i}(t{+}1))\right\rVert^2\nonumber\\
			\overset{\text{\footnotesize Assump.~\ref{assump:bounded-variance}}}&{\leq}\frac{\sigma^2}{n} + \bbE\, \left\lVert\frac{1}{n}\sum\limits_{i{=}1}^n \nabla f_i(\vz_{i}(t{+}1))\right\rVert^2,
		\end{align}
		where the first equality holds due to the unbiasedness of stochastic gradients. Recall from~\eqref{eq:opt-scvx-5} that according to Assumptions~\ref{assump:l-smooth} and~\ref{assump:convex}, we have
		\begin{align}\label{eq:opt-cvx-4}
			\bbE\Bigg\lVert\frac{1}{n}\sum\limits_{i{=}1}^n &\nabla f_i(\vz_{i}(t{+}1))\Bigg\rVert^2\nonumber\\
			\overset{\eqref{eq:opt-scvx-5}}&{\leq}\frac{2L^2}{n}\sum\limits_{i{=}1}^n\bbE\left\lVert \vz_{i}(t{+}1) - \overline{\vx}(t)\right\rVert^2\nonumber\\
			& + 4L\left(\bbE f(\overline{\vx}(t))-f(\vx^\star)\right).
		\end{align}
		Moreover, due to Assumptions~\ref{assump:l-smooth} and~\ref{assump:convex}
		\begin{align}\label{eq:opt-cvx-5}
			\left\langle\nabla f_i(\vz_{i}(t{+}1)),\overline{\vx}(t){-}\vx^\star\right\rangle &\geq  f_i(\overline{\vx}(t))- f_i(\vx^\star)\nonumber\\
			&- \frac{L}{2} \left\lVert\overline{\vx}(t)-\vz_{i}(t{+}1)\right\rVert^2\nonumber \Rightarrow\\
			\frac{1}{n} \sum\limits_{i{=}1}^n \bbE\big\langle\nabla f_i(\vz_{i}(t{+}1)),\overline{\vx}(t)&{-}\vx^\star\big\rangle \geq \bbE f(\overline{\vx}(t))-f(\vx^\star) \nonumber\\
			& - \frac{L}{2n} \underbrace{\bbE\left\lVert\overline{\mX}(t)-\mZ(t{+}1)\right\rVert_F^2}_{\Psi_{x}(t)\text{ in Lemma~\ref{lem:diff-opt}}}.
		\end{align}
		Therefore, according to~\eqref{eq:opt-cvx-1},~\eqref{eq:opt-cvx-2},~\eqref{eq:opt-cvx-4}, and~\eqref{eq:opt-cvx-5}, the following inequality holds:
		\begin{align}\label{eq:convex-base-ineq}
			\bbE\lVert\overline{\vx}(t&{+}1){-}\vx^\star\rVert^2\leq \bbE \lVert\overline{\vx}(t){-}\vx^\star\rVert^2 + \frac{\eta^2\sigma^2}{n}\nonumber\\
			&{-}2\eta(1{-}2L\eta)(\bbE f(\overline{\vx}(t)){-}f(\vx^\star)) + \frac{\eta L}{n}(1{+}2\eta L)\Psi_{x}(t).
		\end{align}
		Considering the average of~\eqref{eq:convex-base-ineq} over $T$ consequent iterations \mbox{$t=0,1,\dots,T{-}1$}, we have
		\begin{align}\label{eq:convex-final-ineq}
			\bbE &f\left(\frac{1}{T}\sum\limits_{t=0}^{T{-}1}\overline{\vx}(t)\right)-f(\vx^\star)
			\leq \frac{1}{T}\sum\limits_{t=0}^{T{-}1}\bbE f(\overline{\vx}(t))-f(\vx^\star)\nonumber\\
			&\leq \frac{\lVert\vx^\star\rVert^2}{2T\eta(1{-}2L\eta)} + \frac{\eta^2 \sigma^2}{2n\eta(1{-}2L\eta)} + \frac{L(1{+}2\eta L)}{2nT(1{-}2\eta L)}\Psi_{x}(t),
		\end{align}
		where by plugging the result of Lemma~\ref{lem:diff-opt} in~\eqref{eq:convex-final-ineq} and \mbox{$\eta=\frac{\sqrt{n}}{4L\sqrt{T}}$}, we conclude the proof of Theorem~\ref{thm:convex}.
	\end{proof}
	
	Theorem~\ref{thm:convex} presents the convergence rate of our algorithm under a milder assumption. We obtain a sublinear convergence rate ($1/\sqrt{T}$) similar to~\cite{taheri2020quantized}, with network and compression dependencies in the faster term ($1/T$). Under the same assumptions, we can see that CHOCO-SGD also has the same convergence rate (see Table~\ref{tab:convex-comparison}). Next, we analyze our algorithm by dropping Assumption~\ref{assump:convex}.

	\begin{theorem}[Smooth and Non-Convex Stochastic Optimization]\label{thm:nonconvex}
		Let the compression operator $Q$ satisfy~\eqref{eq:q-comp}, \mbox{$\mX(0)=\hat\mX(0)\coloneqq\vect{0}$}, and \mbox{$\vy{(0)}= \vect{1}$}. Then, under Assumptions~\ref{assump:bounded-variance}-\ref{assump:l-smooth}, the following property holds for the iterates of 
		update rule~\eqref{eq:update-opt}: for any \mbox{$\omega\in(0,1]$}, and \mbox{$T\geq n$},
		\begin{align*}
			\frac{1}{T}\sum_{t=0}^{T-1} \,\mathbb{E}\left\|\nabla f\left(\frac{1}{n}\sum_{i=1}^n \vx_i(t)\right)\right\|^{2} \leq \frac{C_6}{\sqrt{nT}} + \frac{C_7}{T},
		\end{align*}
		\mbox{$C_6\coloneqq 2L(f(\vect{0}){-}f^\star){+}\sigma^2$}, \mbox{$C_7\coloneqq \frac{3264\,C^2(C{+}1)^2\beta^2(\beta{+}1)^2G^2n}{\omega^2\delta^4\kappa^2}$},\\
		when \mbox{$\eta\coloneqq \frac{\sqrt{n}}{L\sqrt{T}}$}, and \mbox{$\gamma$} as in Lemma~\ref{lem:diff-opt}.
	\end{theorem}
	
	\begin{proof}[Proof of Theorem~\ref{thm:nonconvex}]
		First, note that:
		\begin{align}\label{eq:opt-noncvx-1}
			\overline{\vx}(t{+}1) &= \overline{\vx}(t)  - \frac{\eta}{n} \sum\limits_{i{=}1}^{n} \nabla\tilde{f}_i(\vz_{i}(t{+}1),\vxi_{i,t{+}1}),
		\end{align}
		thus, by Assumption~\ref{assump:l-smooth}, we have
		\begin{align}\label{eq:opt-noncvx-2}
			\bbE f(\overline{\vx}(t{+}1)) &\leq \bbE f(\overline{\vx}(t)) \nonumber\\
			&-\frac{\eta}{n} \bbE\left\langle\nabla f(\overline{\vx}(t)), \sum\limits_{i{=}1}^{n}\nabla\tilde{f}_i(\vz_{i}(t{+}1),\vxi_{i,t{+}1})\right\rangle\nonumber\\
			&+ \frac{\eta^2 L}{2} \bbE \left\lVert\frac{1}{n}\sum\limits_{i{=}1}^{n}\nabla\tilde{f}_i(\vz_{i}(t{+}1),\vxi_{i,t{+}1})\right\rVert^2\nonumber\\
			&\leq \bbE f(\overline{\vx}(t)) \nonumber\\
			&-\frac{\eta}{n} \bbE\left\langle\nabla f(\overline{\vx}(t)), \sum\limits_{i{=}1}^{n}\nabla f_i(\vz_{i}(t{+}1))\right\rangle\nonumber\\
			&+ \frac{\eta^2 L}{2} \bbE \left\lVert\frac{1}{n}\sum\limits_{i{=}1}^{n}\nabla\tilde{f}_i(\vz_{i}(t{+}1),\vxi_{i,t{+}1})\right\rVert^2.
		\end{align}
		Moreover, the following inequality holds:
		\begin{align}\label{eq:opt-noncvx-3}
			\bbE&\Bigg\langle\nabla f(\overline{\vx}(t)), \frac{1}{n}\sum\limits_{i{=}1}^{n}\nabla f_i(\vz_{i}(t{+}1))\Bigg\rangle\nonumber\\
			&= \frac{1}{2} \bbE \left\lVert\nabla f(\overline{\vx}(t))\right\rVert^2 + \frac{1}{2} \bbE \left\lVert\frac{1}{n}\sum\limits_{i{=}1}^{n}\nabla f_i(\vz_{i}(t{+}1))\right\rVert^2\nonumber\\
			&-\frac{1}{2} \bbE \left\lVert\nabla f(\overline{\vx}(t)) - \frac{1}{n}\sum\limits_{i{=}1}^{n}\nabla f_i(\vz_{i}(t{+}1))\right\rVert^2\nonumber\\
			&= \frac{1}{2} \bbE \left\lVert\nabla f(\overline{\vx}(t))\right\rVert^2 + \frac{1}{2} \bbE \left\lVert\frac{1}{n}\sum\limits_{i{=}1}^{n}\nabla f_i(\vz_{i}(t{+}1))\right\rVert^2\nonumber\\
			&-\frac{1}{2} \bbE \left\lVert\frac{1}{n}\sum\limits_{i{=}1}^{n}\left[\nabla f_i(\overline{\vx}(t)) - \nabla f_i(\vz_{i}(t{+}1))\right]\right\rVert^2\nonumber\\
			&\geq\frac{1}{2} \bbE \left\lVert\nabla f(\overline{\vx}(t))\right\rVert^2 + \frac{1}{2} \bbE \left\lVert\frac{1}{n}\sum\limits_{i{=}1}^{n}\nabla f_i(\vz_{i}(t{+}1))\right\rVert^2\nonumber\\
			&-\frac{1}{2n}\sum\limits_{i{=}1}^{n}\bbE \left\lVert\nabla f_i(\overline{\vx}(t)) - \nabla f_i(\vz_{i}(t{+}1))\right\rVert^2\nonumber\\
			\overset{\text{\footnotesize Assump.~\ref{assump:l-smooth}}}&{\geq}\frac{1}{2} \bbE \left\lVert\nabla f(\overline{\vx}(t))\right\rVert^2 + \frac{1}{2} \bbE \left\lVert\frac{1}{n}\sum\limits_{i{=}1}^{n}\nabla f_i(\vz_{i}(t{+}1))\right\rVert^2\nonumber\\
			&-\frac{L^2}{2n}\sum\limits_{i{=}1}^{n}\bbE\left\lVert\overline{\vx}(t) - \vz_{i}(t{+}1)\right\rVert^2.
		\end{align}
		Hence, according to~\eqref{eq:opt-cvx-2},~\eqref{eq:opt-noncvx-2},~\eqref{eq:opt-noncvx-3}, we have
		\begin{align}\label{eq:nonconvex-base-ineq}
			\bbE\lVert\nabla& f(\overline{\vx}(t))\rVert^2 + (1{-}\eta L) \bbE \left\lVert\frac{1}{n}\sum\limits_{i{=}1}^{n}\nabla f_i(\vz_{i}(t{+}1))\right\rVert^2\nonumber\\
			&\leq \frac{2\bbE f(\overline{\vx}(t))-2\bbE f(\overline{\vx}(t{+}1))}{\eta}
			+ \frac{\eta\,\sigma^2 L}{n} + \frac{2L^2}{n}\Psi_{x}(t).
		\end{align}
		Therefore, with a similar approach to~\eqref{eq:convex-final-ineq}, we obtain
		\begin{align}\label{eq:nonconvex-final-ineq}
			\frac{1}{T}\sum_{t=0}^{T-1}&\left[\bbE\lVert\nabla f(\overline{\vx}(t))\rVert^2 + (1{-}\eta L) \bbE \left\lVert\frac{\sum\limits_{i{=}1}^{n}\nabla f_i(\vz_{i}(t{+}1))}{n}\right\rVert^2\right]\nonumber\\
			&\leq \frac{2\bbE f(\overline{\vx}(0))-2\bbE f(\overline{\vx}(T))}{\eta T} + \frac{\eta\,\sigma^2 L}{n}\nonumber\\
			&+ \frac{3264\,C^2 (C{+}1)^2\beta^2(\beta{+}1)^2 G^2\eta^2L^2}{\omega^2\delta^4\kappa^2},
		\end{align}
		which concludes the proof of Theorem~\ref{thm:nonconvex} when \mbox{$\eta=\frac{\sqrt{n}}{L\sqrt{T}}$}.
	\end{proof}

	Theorem~\ref{thm:nonconvex} suggests a sublinear convergence rate to a first-order stationary point in the non-convex regime. Table~\ref{tab:nonconvex-comparison} summarize the comparison between~\cite{koloskova2019decentralized2,taheri2020quantized}, and ours. Next, we show an outline of the proof.
	

	\section{Numerical Experiments}\label{sec:experiments}
	In this section, we verify the performance of our proposed algorithm through two sets of numerical experiments. We first consider the decentralized average consensus problem and show that our algorithm can achieve convergence under arbitrary compression. Then, we validate the communication efficiency of our algorithm on a decentralized logistic regression problem.

	\subsection{Consensus}\label{sec:experiments-cons}
	We first consider an average consensus problem with \mbox{$d=300$} parameters over a directed Ring graphs with different number of agents \mbox{$n\in\{20,50,100,200,500\}$}. We consider a grid over \mbox{$(0,1]$} for compression ratio $\omega$, thus \mbox{$\mathrm{top}_{100\omega\%}$}~\cite{toghani2021communication} as the proper compression operator. We quantify the number of round required for each pair \mbox{$(n,\omega)$}, to reach an $\varepsilon$-accuracy where \mbox{$\varepsilon{=}1\cdot10^{-5}$}. We compare the performance of our algorithm compared to~\cite{taheri2020quantized}. To have a fair comparison, we do not fine-tune \mbox{$\gamma$}, and simply select \mbox{$\gamma=\omega$} for this experiment. 
	
	Given the described setup, Figure~\ref{fig:cons-omega} shows the number of rounds required for each algorithm to reach an $\varepsilon$-consensus. As depicted in Figure~\ref{fig:cons-omega}, each solid line shows the number of rounds required for $n$ agents to reach consensus on a directed Ring with $\omega$-compressed messages. This figure shows the importance of consensus stepsize for the algorithm to reach consensus under any arbitrary compression ratio $\omega$.
	
	\begin{figure}[!t]
		\centering
		\includegraphics[width=0.99\linewidth]{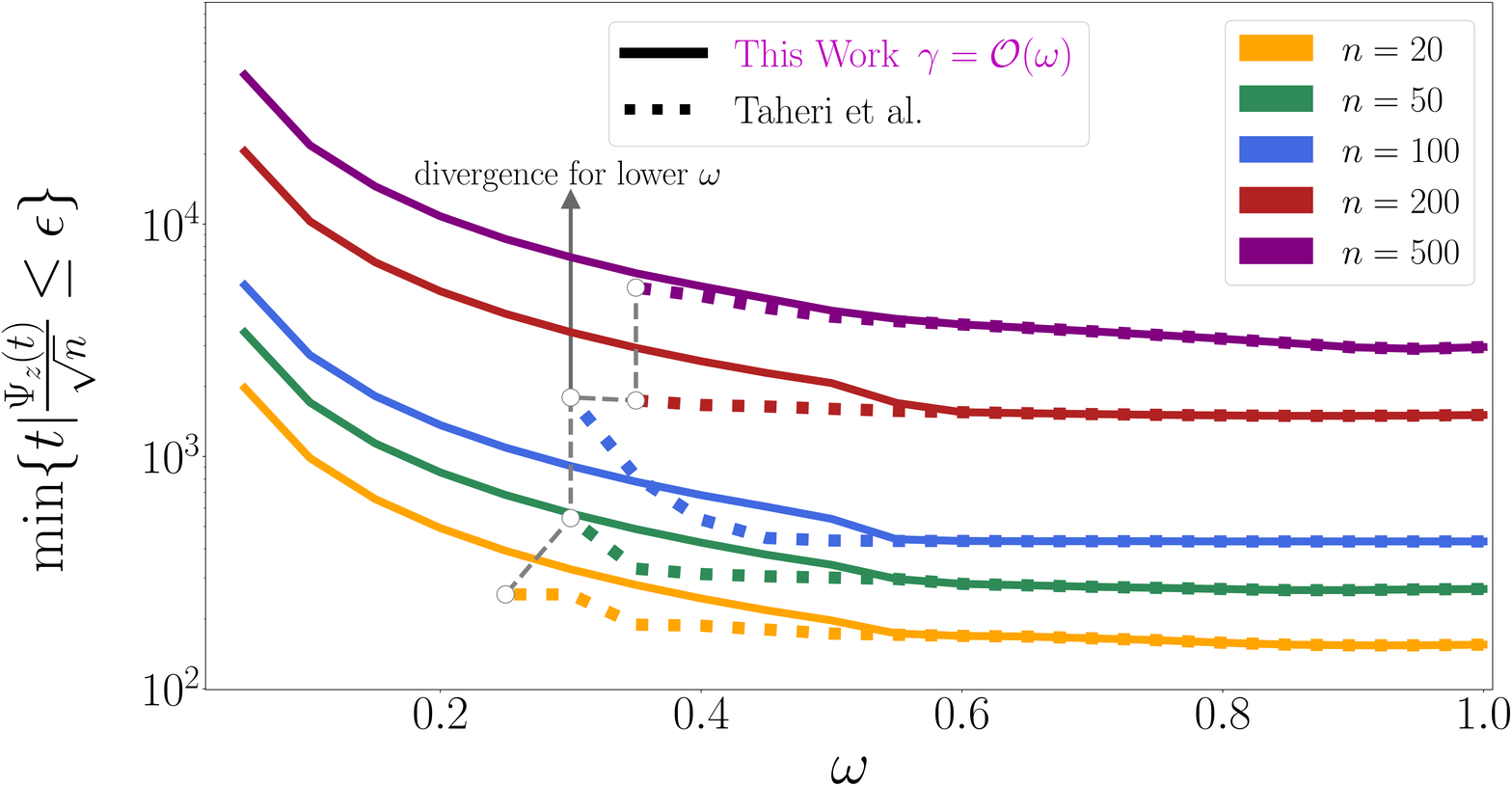}
		\put(-98,111.2){{\tiny\cite{taheri2020quantized}}}
		\caption{\textbf{Arbitrary compression ratio:} The number of rounds required for each algorithm to reach an \mbox{$\varepsilon$-accuracy} \mbox{($\varepsilon{=}10^{{-}5}$)} for an average consensus problem with \mbox{$d=300$} over directed Ring graphs with $n$ agents using the compression operator \mbox{$\mathrm{top}_{100\omega\%}$}. For compression ratios \mbox{$\omega\in(0,1]$}, we compare~\cite{taheri2020quantized} and our compressed push-sum consensus with \mbox{$\gamma = \mcO(\omega)$}. Each line associates with a fixed $n$. For small compression ratios $\omega$, the method in~\cite{taheri2020quantized} is not guaranteed to converge.}
		\label{fig:cons-omega}
	\end{figure}

	\subsection{Regularized Logistic Regression}\label{sec:experiments-logreg}
	
	\begin{figure*}[!ht]
		\centering
		\begin{minipage}{0.9\textwidth}
			\hspace{-2em}
			\includegraphics[width=0.5\linewidth]{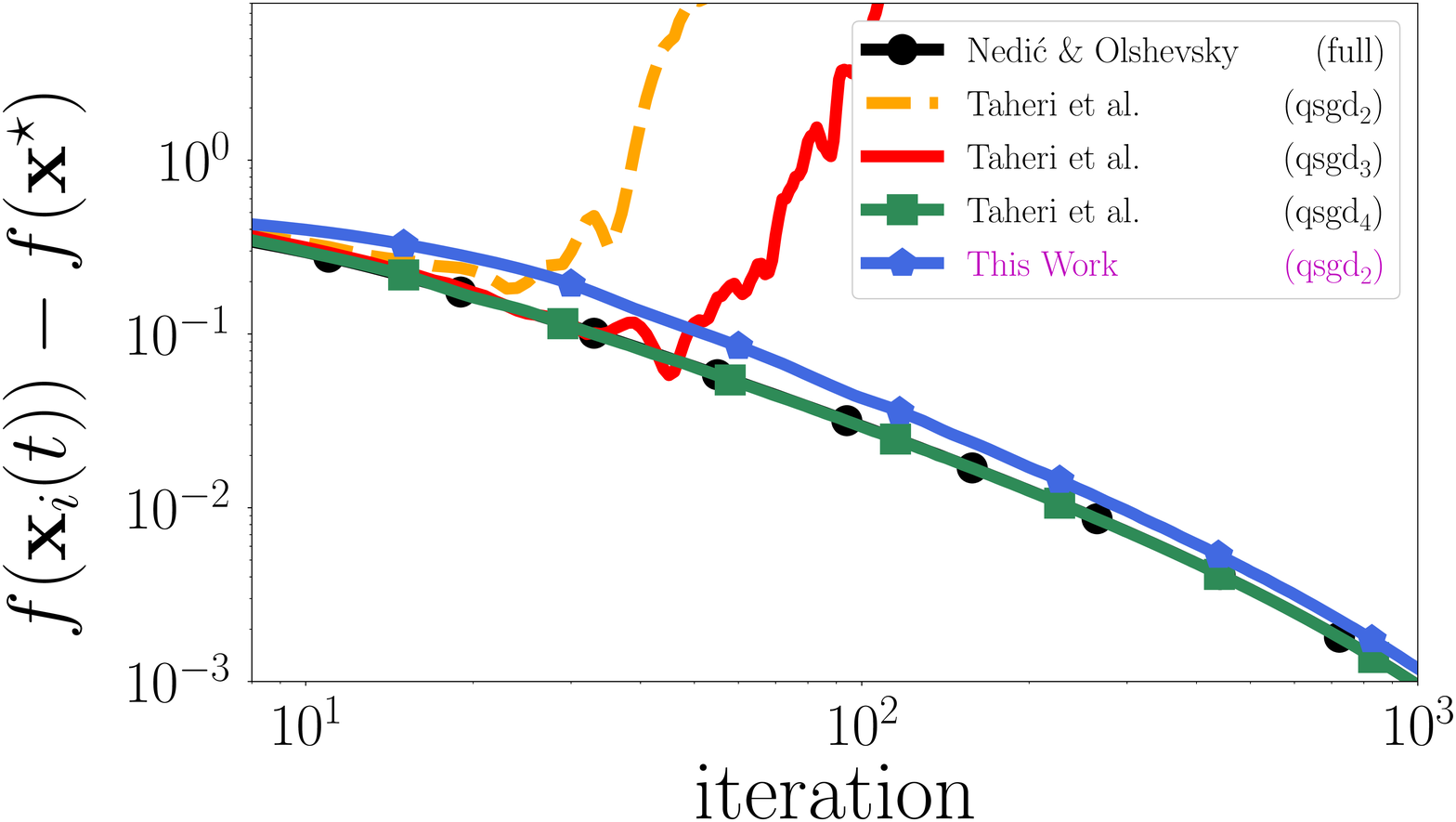}
			\put(-37,123){{\tiny\cite{nedic2016stochastic}}}
			\put(-52,114.8){{\tiny\cite{taheri2020quantized}}}
			\put(-52,106.5){{\tiny\cite{taheri2020quantized}}}
			\put(-52,98.4){{\tiny\cite{taheri2020quantized}}}
			\hspace{2em}
			\includegraphics[width=0.5\linewidth]{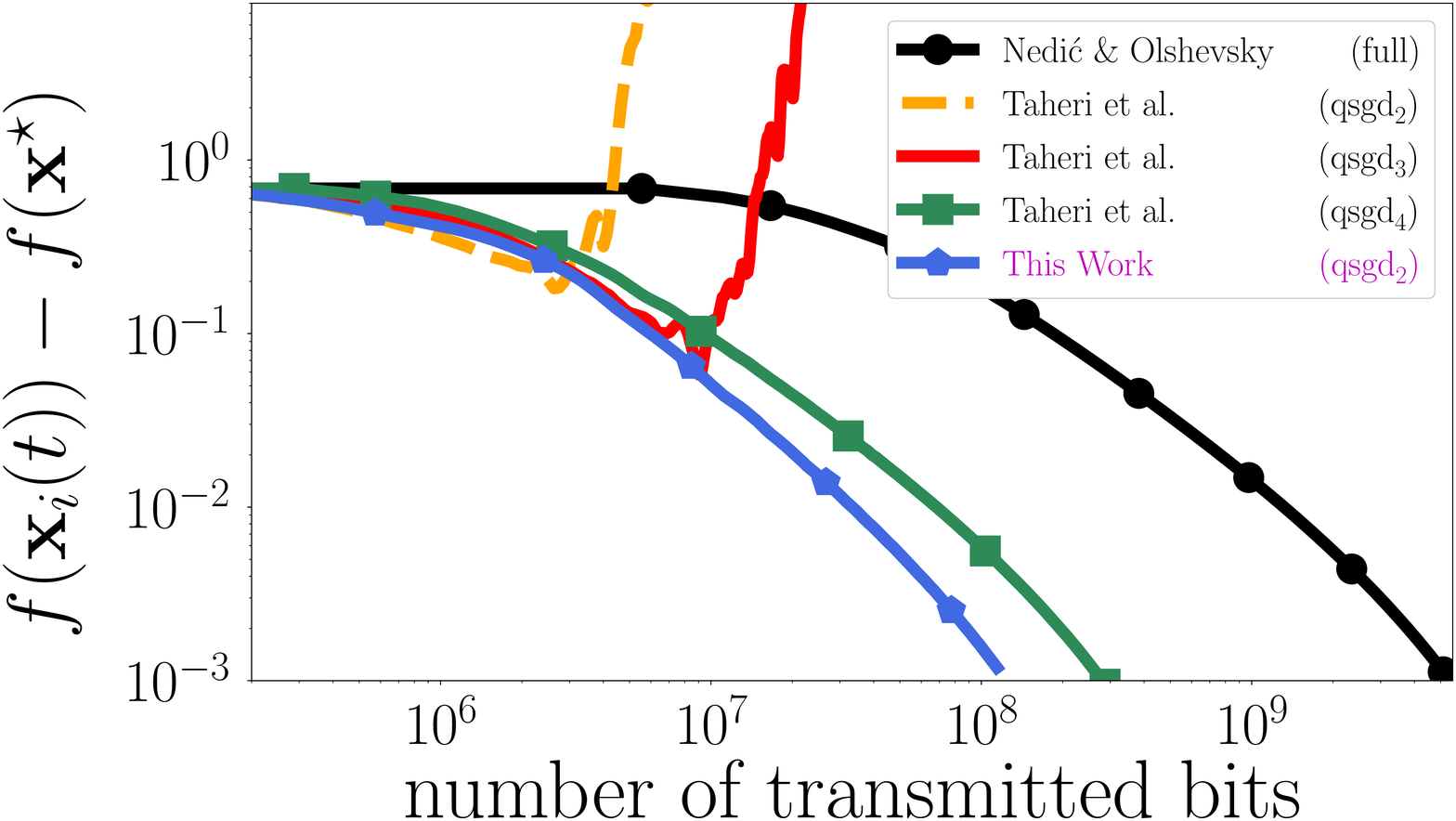}
			\put(-32,123){{\tiny\cite{nedic2016stochastic}}}
			\put(-46,114.8){{\tiny\cite{taheri2020quantized}}}
			\put(-46,106.5){{\tiny\cite{taheri2020quantized}}}
			\put(-46,98.4){{\tiny\cite{taheri2020quantized}}}
		\end{minipage}
		\caption{\textbf{Communication-Efficiency and Arbitrary Compression:} Each simulation is the average of 10 runs. We consider a decentralized logistic regression problem in~\eqref{eq:log-reg-problem}, with \mbox{$n=100$} agents, \mbox{$d=200$} dimensional parameters, and \mbox{$m=20$} data samples, over a static, directed, and strongly connected Erd\H{o}s-R\'enyi graph with probability $(\log n)/n$. We compare the performance of our algorithm given \mbox{$\mathrm{qsgd}_{2}$} with~\cite{nedic2016stochastic} (no compression), and~\cite{taheri2020quantized} using \mbox{$\mathrm{qsgd}_{k}$}, for $k=2,3,4$. The loss curve of one agent \mbox{$i\in[n]$} is shown based on the number of \textbf{(left)} iterations, and \textbf{(right)} transmitted bits, across the network.}
		\label{fig:logistic-regression-iteration-bit}
	\end{figure*}
	
	Now, we consider a decentralized logistic regression problem with $\ell_2$ regularization loss as follows: 
	\begin{align}\label{eq:log-reg-problem}
		\min_{\vx\in\bbR^d} \bigg[f(\vx) &\coloneqq \frac{1}{n} \sum_{i = 1}^n f_i(\vx)\bigg],\nonumber\\
		f_i(\vx) \coloneqq \frac{1}{m} \sum_{r = 1}^m \log(1 &{+} \exp(-b_{ir} \va_{ir}^\top \vx)) + \frac{1}{2mn} \lVert \vx\rVert_2^2,
	\end{align}
	with $m$ (possibly non-iid) local data samples at each node \mbox{$i\in[n]$}, where \mbox{$a_{ir} \in \bbR^d$} and $b_{ir} {\in} \{-1, 1\}$ respectively denote the features and label of the $r$-th sample at node $i$. We consider a binary classification task on a synthetic dataset of two separable high-dimensional Cones. We also consider a static, directed, and strongly connected Erd\H{o}s-R\'enyi graph with connection probability \mbox{$(\log n)/n$} as the communication network. Note that by directed Erd\H{o}s-R\'enyi, we mean that for each two nodes \mbox{$i,j\in[n]$}, a link from $i$ to $j$ exists (independent of other links) with some probability (in this case \mbox{$(\log n)/n$}). Also, note that we select a realization of this class of graphs which is strongly connected. We consider a set of \mbox{$n=100$} agents with \mbox{$d{=}200$} dimensional parameters, and \mbox{$m=30$} local samples at each node. We also consider \mbox{$\mathrm{qsgd}_{k}$}~\cite{alistarh2017qsgd} as the compression operator. Similar to the previous example, we do not fine-tune $\gamma$ and simply select it to be \mbox{$\mcO(\omega)$}.
	
	We compare the performance of our algorithm with methods in~\cite{nedic2016stochastic,taheri2020quantized}. On the one hand, the algorithm in~\cite{nedic2016stochastic} has no compression module. On the other hand, the compressed gradient-push in~\cite{taheri2020quantized} does not converge for any arbitrary compression ratio. We therefore consider our algorithm with \mbox{$\mathrm{qsgd}_{2}$} and \mbox{$\gamma=\omega$}, as well as the method in~\cite{taheri2020quantized} with \mbox{$\mathrm{qsgd}_{k}$}, where \mbox{$k\in\{2,3,4\}$} is the precision level of the quantizer. For precision levels \mbox{$k<4$}, the method in~\cite{taheri2020quantized} does not converge. Figure~\ref{fig:logistic-regression-iteration-bit} shows the suboptimality loss of these methods given the number of iterations and transmitted bits. In this problem,~\cite{taheri2020quantized} converges for \mbox{$k=4$}, while our algorithm converges with \mbox{$k=2$}. The figure on the right-hand side shows that our algorithm converges with fewer communication bits without any fine-tuning on $\gamma$.
	

	\section{Conclusions}\label{sec:conclusion}
	This work studied decentralized consensus and stochastic optimization over a fixed, directed, and strongly connected network. Revisiting~\cite{koloskova2019decentralized,taheri2020quantized}, we proposed an algorithm with guaranteed convergence under any compression ratio \mbox{$\omega\in(0,1]$}, and appropriate assumptions. We further presented the theoretical guarantees for our algorithm under standard assumptions on three smooth function classes: (i) strongly-convex, (ii) convex, and (iii) non-convex. We also showed empirical analysis that illustrates the arbitrary compression and communication efficiency of the proposed method. Extensions and results to time-varying networks, scalability to the number of agents, and robustness to adversarial scenarios remain as future work.


	\bibliographystyle{IEEEbib}
	\bibliography{ref}

\end{document}